%% file: sample-sigconf.tex
\documentclass[sigconf]{acmart}
\AtBeginDocument{%
  }

\copyrightyear{2026}
\acmYear{2026}
\setcopyright{cc}
\setcctype{by}
\acmConference[KDD '26]{Proceedings of the 32nd ACM SIGKDD Conference on Knowledge Discovery and Data Mining V.2}{August 09--13, 2026}{Jeju Island, Republic of Korea}
\acmBooktitle{Proceedings of the 32nd ACM SIGKDD Conference on Knowledge Discovery and Data Mining V.2 (KDD '26), August 09--13, 2026, Jeju Island, Republic of Korea}
\acmDOI{10.1145/3770855.3817843}
\acmISBN{979-8-4007-2259-2/2026/08}

\usepackage{algorithm}
\usepackage{algorithmic}
\usepackage{enumitem}
\usepackage{graphicx}

\usepackage{microtype}
\input{math.tex}
\begin{document}

\title{Decentralized Stochastic Nonconvex Optimization under~the~$(L_0,L_1)$-Smoothness}


\author{Luo Luo}
\affiliation{%
  \institution{School of Data Science, Fudan University}
  \city{Shanghai}
  \country{China}}
\email{luoluo@fudan.edu.cn}

\author{Xue Cui}
\affiliation{%
  \institution{School of Data Science, Fudan University}
  \city{Shanghai}
  \country{China}}
\email{xcui24@m.fudan.edu.cn}

\author{Tingkai Jia}
\affiliation{%
  \institution{East China Normal University}
  \city{Shanghai}
  \country{China}}
\email{51275902086@stu.ecnu.edu.cn}

\author{Cheng Chen}
\authornote{The corresponding author}
\affiliation{%
  \institution{East China Normal University}
  \city{Shanghai}
  \country{China}}
\email{chchen@sei.ecnu.edu.cn}

\renewcommand{\shortauthors}{Luo, Cui, Jia, and Chen}

\begin{abstract}
This paper focuses on the decentralized stochastic optimization problem  $f(\mathbf{x})=\frac{1}{m}\sum_{i=1}^m f_i(\mathbf{x})$ over a connected network of $n$ agents, 
where each local function has the form of $f_i(\mathbf{x}) = {\mathbb E}\left[F(\mathbf{x};{\boldsymbol \xi}_i)\right]$ which satisfies the $(L_0,L_1)$-smooth condition but possibly nonconvex and each random variable ${\boldsymbol \xi}_i$ follows distribution ${\mathcal D}_i$.
We propose a novel algorithm called decentralized normalized stochastic gradient descent (DNSGD), which can achieve an $\epsilon$-stationary point at each local agent.
We present a new framework for analyzing decentralized first-order methods in the $(L_0,L_1)$-smooth setting, based on the Lyapunov function related to the product of the gradient norm and the consensus error.
We show that the proposed algorithm attains the upper bounds on the sample complexity of ${\mathcal O}(m^{-1}(L_f\sigma^2\Delta_f\epsilon^{-4} + \sigma^2\epsilon^{-2} + L_f^{-2}L_1^3\sigma^2\Delta_f\epsilon^{-1} + L_f^{-2}L_1^2\sigma^2))$ per agent
and the communication complexity of $\tilde{\mathcal O}((L_f\epsilon^{-2} + L_1\epsilon^{-1})\gamma^{-1/2}\Delta_f)$, 
where $L_f=L_0 +L_1\zeta$, $\sigma^2$ is the variance of the stochastic gradient, $\Delta_f$ is the initial optimal function value gap, $\gamma$ is the spectral gap of the network, and $\zeta$ is the degree of the gradient dissimilarity. 
In the special case of $L_1=0$, the above results (nearly) match the lower bounds of decentralized stochastic nonconvex optimization under the standard smoothness.
We also conduct numerical experiments to show the empirical superiority of our method.

\end{abstract}

\ccsdesc[500]{Mathematics of computing~Nonconvex optimization}

\keywords{Decentralized stochastic optimization, Relaxed smoothness, Nonconvex optimization}


\maketitle

\section{Introduction}

This paper considers the decentralized stochastic optimization over a connected network of $n$ nodes. Specifically, we study the minimization problem
\begin{align*}
    \min_{\vx\in\BR^d} f(\vx)=\frac{1}{m}\sum_{i=1}^m f_i(\vx),
\end{align*}
where $f_i:\BR^d\to\BR$ is the local function on $i$th agent and has the form of 
\begin{align*}
f_i(\vx) = \BE\left[F_i(\vx;\vxi_i)\right].    
\end{align*}
We are interested in the setting that each $f_i$ is $(L_0,L_1)$-smooth, also known as relaxed smooth~\citep{zhang2020gradient}, but possibly nonconvex and the local functions are heterogeneous, i.e., each random variable $\vxi_i$ follows the distribution $\fD_i$. 
The optimization theory and algorithms for relaxed smooth problem have received significant attention in recent years \citep{chen2023generalized,vankov2024optimizing,li2024problem,li2023convergence,gaash2025convergence,gorbunov2024methods,koloskova2023revisiting,chezhegov2025convergence,tovmasyan2025revisiting,crawshaw2025complexity,lobanov2024linear,tyurin2024toward,reisizadeh2025variance,xie2024trust,zhang2020improved,yu2025mirror,khirirat2024error,khirirat2024communication,jiang2025decentralized,borodich2025nesterov}, 
since the empirical studies suggest the relaxed smoothness can better characterize the training of neural network models than the classical smoothness \citep{zhang2020gradient,cooper2024empirical,crawshaw2022robustness}.
However, most existing works focus on the centralized setting, and how to well address the relaxed smooth optimization problem in the decentralized scenarios remains unclear.

The main challenge in decentralized optimization is that each agent can only access the information from itself or its neighbors, which leads to the consensus error among local variables \citep{nedic2018network}.
For the standard smooth optimization, the gradient tracking technique successfully uses the gradient Lipschitz continuity to reduce the consensus error and guarantees the convergence of the algorithms \citep{nedic2017achieving,qu2017harnessing}.
Additionally, the multi-consensus steps with Chebyshev acceleration \citep{liu2011accelerated,arioli2014chebyshev} can further improve the communication complexity, leading to the optimal decentralized methods for both convex and nonconvex cases \citep{song2024optimal,xin2020general,hendrikx2021optimal,li2022variance,metelev2024decentralized,liu2024decentralized,luo2022complexity,ye2023multi,bai2024complexity,yuan2022revisiting,lu2021optimal}.
For nonsmooth optimization, the recent works focus on finding the sub-optimal solutions or the approximate (generalized) stationary points of the Lipschitz continuous objective \citep{sahinoglu2024online,lin2023decentralized,kovalev2024lower,lan2020communication}, establishing efficient decentralized stochastic algorithms based on the boundedness of local (sub)gradients.

In this paper, we focus on decentralized optimization under the relaxed smoothness, which is more challenging than above classical settings since neither the local functions nor their gradients enjoy the Lipschitz continuity.
We propose a novel algorithm called decentralized normalized stochastic gradient descent (DNSGD) to address above issues.
Our method does not include the gradient clipping, avoiding the potential inconsistency in the update modes of different local variables arising from the consensus errors.
We further construct a new Lyapunov function related to the product
of the gradient norm and the consensus error, which successfully characterizes the 
convergence of our decentralized method in the relaxed smooth setting that allows local gradients to be neither bounded nor Lipschitz continuous.
Moreover, we conduct experiments on the application of image classification to show the empirical advantage of our proposed method in practice.

\section{Preliminaries}

We use the bold lowercase letter with the subscript to represent the local variable, e.g., vector $\vx_i\in \mathbb{R}^{d}$ is a local variable on the $i$th agent. 
We use the bold uppercase letter to present the aggregated matrix consisting of local variables from all of the $m$ agents, e.g., 
\mbox{$\mathbf{X} = [\mathbf{x}_1, \dots, \mathbf{x}_m ]^\top\in\mathbb{R}^{m \times d}$, where $\vx_i\in \mathbb{R}^{d}$}.
We also denote the mean vector by the bold lowercase letter with bar,  e.g., $\bar{\vx} = \frac{1}{m} \vone^\top\mX\in\BR^{1\times d}$, 
where~$\vone=[1,\dots,1]^\top\in\BR^{m}$.
In addition, the notation $\| \cdot \|$ represents the Euclidean norm (or the Frobenius) norm for given vector (or matrix),
and we use $\mI$ to present the identity matrix.
For the ease of presentation, we allow the input of the function be organized as either column vector or row vector.

We impose following common assumptions for nonconvex optimization under the $(L_0, L_1)$ smoothness, a.k.a. relaxed smoothness.

\begin{asm}\label{asm:lower}
    We suppose the objective $f$ is lower bounded, i.e., there exists some $f^*\in\BR$ such that
    \begin{align*}
    f^* = \inf_{\vx\in\BR^d} f(\vx) > -\infty.
    \end{align*}
\end{asm}

\begin{asm}\label{asm:sym-1}
    We suppose each local function $f_i$ is $(L_0,L_1)$-smooth i.e., there exists some $L_0,L_1\geq0$ such that  
    \begin{align}\label{eq:asm:sym-1}
    \Norm{\nabla f_i(\vx)-\nabla f_i(\vy)} \leq \left(L_0 + L_1 \Norm{\nabla f_i(\vx)}\right)\Norm{\vx-\vy}    
    \end{align}
    for all $i\in[m]$ and $\vx,\vy\in\BR^d$ with $\Norm{\vx-\vy}\leq 1/L_1$, where we define $1/L_1=+\infty$ if~$L_1=0$.
\end{asm}

\begin{remark}
The notion of relaxed smoothness was first introduced by \citet{zhang2020gradient}, which requires the second-order differentiability. 
Here, we consider the weaker condition (Assumption~\ref{asm:sym-1}) characterized by the gradient, which is provided by \citet{zhang2020improved}.
\end{remark}

We also introduce the assumption of the gradient dissimilarity to describe the heterogeneity of the local functions, which is widely used in distributed optimization \citep{woodworth2020minibatch,gorbunov2021local,assran2019stochastic,lian2017can}.

\begin{asm}[bounded gradient dissimilarity]\label{asm:heterogeneous}
    We suppose there exists some $\zeta\geq 0$ such that the local gradient satisfies
    \begin{align*}
        \Norm{\nabla f_i(\vx) - \nabla f(\vx)} \leq \zeta
    \end{align*}
    for all $i\in[m]$ and $\vx\in\BR^d$.
\end{asm}

The following proposition shows that the objective $f$ is relaxed smooth if the local functions are relaxed smooth with bounded dissimilarity. 

\begin{prop}\label{prop:global-l0l1}
Under Assumptions \ref{asm:sym-1} and \ref{asm:heterogeneous}, the objective function $f=\frac{1}{m}\sum_{i=1}^m f_i$ is $(L_f,L_1)$-smooth with $L_f=L_0+L_1\zeta$, 
i.e., we have
\begin{align}\label{eq:global-relax-smooth}
    \Norm{\nabla f(\vx)-\nabla f(\vy)} \leq \left(L_f + L_1 \Norm{\nabla f(\vx)}\right)\Norm{\vx-\vy}    
\end{align}
for all $\vx,\vy\in\BR^d$ with $\Norm{\vx-\vy}\leq 1/L_1$.
\end{prop}

We are interested in decentralize stochastic first-order methods, which can access the stochastic first-order oracle $\nabla F(\vx;\vxi_i)$ that satisfies the following assumption.

\begin{asm}\label{asm:sfo}
For given $i\in[m]$ and $\vx\in\BR^d$, we suppose the $i$th agent can draw $\vxi_i\sim\fD_i$ and access the stochastic first-order oracle $\nabla F_i(\vx;\vxi_i)$ which satisfies
\begin{align*}
    \BE[\nabla F_i(\vx;\vxi_i)] = \nabla f_i(\vx)
    \quad\text{and}\quad
    \BE\left[\Norm{\nabla F_i(\vx;\vxi_i) - \nabla f_i(\vx)}^2\right] \leq \sigma^2
\end{align*}
for some $\sigma> 0$.    
\end{asm}

In decentralized setting, each agent can only exchange information with its neighbors in per communication round. 
One communication step in decentralized optimization typically can be presented by multiplication on the mixing matrix~$\mW \in \BR^{m \times m}$, which satisfies the following standard assumption \citep{scaman2017optimal,hendrikx2021optimal,ye2023multi}.

\begin{asm}\label{ams:W}
Let $\mW\in \mathbb{R}^{m \times m}$ be the mixing matrix associated with the connected and undirected network. 
We suppose
(a) the matrix $\mW$ is symmetric with $w_{i,j} \ge 0$ for all $i,j\in[m]$, and $w_{i,j} \neq 0$ if and only if the $i$th agent and the $j$th agent are connected or $i = j$;
(b) It holds $\mathbf{0} \preceq \mW \preceq \mI$, $\mW^\top \mathbf{1}=\mW \mathbf{1} = \mathbf{1}$, and $\operatorname{null}(\mI - \mW) = \operatorname{span}(\mathbf{1})$.
\end{asm}

Note that Assumption~\ref{ams:W} implies the mixing matrix $\mW$ is doubly stochastic, which indicates that multiplying by $\mW$ on the aggregated matrix encourages all local variables converge to their average.
We use $\lambda_2(\mW)$ to present the second largest eigenvalue of the mixing matrix $\mW$. Under Assumption~\ref{ams:W}, we have $\lambda_2(\mW)\in[0,1)$, then we define the spectral gap of $\mW$ as $\gamma=1-\lambda_2(\mW)$.
Additionally, we present the multi-consensus steps with Chebyshev acceleration in Algorithm \ref{alg:fm} \citep{liu2011accelerated,arioli2014chebyshev}, which encourage all local variables to converge to the average as follows.

\begin{prop}[{\citet[Proposition 1]{ye2023multi}}]
\label{proposition_fastmix}
    Under Assumption \ref{ams:W}, the output of Algorithm~\ref{alg:fm} holds $\frac{1}{m} \mathbf{1}^\top \mY^{(K)} = \bar{\vy}^{(0)}$ and
    \begin{align*}
    \big\|\mY^{(K)} - \mathbf{1} \bar{\vy}^{(0)} \big\| \leq c_1 \big(1 - c_2 \sqrt{1-\lambda_2(\mW)}\,\big)^K \big\| \mY^{(0)} - \mathbf{1} \bar{\vy}^{(0)}\big\|,     
    \end{align*}    
    where $\bar{\vy}^{(0)} = \frac{1}{m} \mathbf{1}^\top \mY^{(0)}$, $c_1 = \sqrt{14}$, and $c_2 = 1 - 1/\sqrt{2}$.
\end{prop}

\begin{algorithm}[t]
	\caption{$\AG(\mY^{(0)},\mW, K)$} \label{alg:fm}
	\begin{algorithmic}[1]
		\STATE $\mY^{(-1)}=\mY^{(0)}$ \\[0.1cm]
        \STATE $\eta_y=(1-\sqrt{1-(\lambda_2(\mW))^2}\,)/(1+\sqrt{1-(\lambda_2(\mW))^2}\,)$ \\[0.1cm]
		\STATE \textbf{for} $k = 0, 1, \dots, K$ \textbf{do}\\[0.1cm]
		\STATE\quad $\mY^{(k+1)}=(1+\eta_y)\mW\mY^{(k)}-\eta_y\mY^{(k-1)}$ \\[0.1cm]
		\STATE\textbf{end for} \\[0.1cm]
		\STATE \textbf{Output:} $\mY^{(K)}$
	\end{algorithmic}
\end{algorithm}

\section{The Algorithm and Main Results}

\begin{algorithm}[t]
	\caption{DNSGD} \label{alg:DNSGD}
	\begin{algorithmic}[1]
		\STATE \textbf{Input:} $\bx^0\in\BR^{1\times d}$, $\mW\in\BR^{m\times m}$, $\eta>0$, $\hat K,K,T,b\in\BN$ \\[0.15cm]
        \STATE $\mX^0=\vone\bx^0$;\\
        \STATE i.i.d sample $\vxi_{i,k}^0\sim\fD_i$ for all $k\in[b]$ at each agent $i$  \\[0.12cm]
        \STATE \label{line:G0} $\mG^0 = \begin{bmatrix}
              \displaystyle{\frac{1}{b}\sum_{k=1}^{b} \nabla F_1(\bx^0;\vxi_{1,k}^0)} 
              &\!\!\!\!\dots\!\!\!\! & 
              \displaystyle{\frac{1}{b}\sum_{k=1}^{b} \nabla F_m(\bx^0;\vxi_{m,k}^0)}
            \end{bmatrix}^\top$ \\[0.15cm]
        \STATE\label{line:V0} $\mV^0 = \AG\big(\mG^0, \mW, \hat K\big)$ \\[0.15cm]        
        \STATE \textbf{for} $t = 0, 1, \dots, {T-1}$ \textbf{do} \\[0.12cm]   		     
        \STATE\quad\label{line:U} $\mU^t = \begin{bmatrix}
            \dfrac{\vv^t_1}{\Norm{\vv^t_1}} & \dots & \dfrac{\vv^t_m}{\Norm{\vv^t_m}}
        \end{bmatrix}^\top$ \label{line:update-V} \\[0.15cm]
        \STATE\quad $\mX^{t+1} = \AG\left(\mX^{t} - \eta\mU^t, \mW, K\right)$ \label{line:update-X} \\[0.15cm]
        \STATE\quad i.i.d sample $\vxi_{i,k}^{t+1}\sim\fD_i$ for all $k\in[b]$ at each agent $i$
        \STATE\quad \label{line:Gt} $\mG^{t+1} = \begin{bmatrix}
              \displaystyle{\frac{1}{b}\sum_{k=1}^b \nabla F_1(\vx_1^{t+1};\vxi_{1,k}^{t+1})} &\!\!\!\!\dots\!\!\!\! & 
              \displaystyle{\frac{1}{b}\sum_{k=1}^b \nabla F_m(\vx_m^{t+1};\vxi_{m,k}^{t+1})}
            \end{bmatrix}^\top$ \\[0.15cm]
        \STATE\quad \label{line:Vt1} $\mV^{t+1} = \AG\left(\mV^{t} + \mG^{t+1} - \mG^{t}, \mW, K\right)$ \\[0.15cm]       
        \STATE\textbf{end for} \\[0.15cm]
        \STATE $\hat\vx_i\sim{\rm Unif}\{\vx_i^0,\dots,\vx_i^{T-1}\}$ for each agent $i$\\[0.15cm] 
		\STATE \textbf{Output:} $\hat\vx_i$ for each agent $i$
    \end{algorithmic}
\end{algorithm}

We propose decentralized normalized stochastic gradient
descent (DNSGD) in Algorithm~\ref{alg:DNSGD}, which combines the techniques of normalized gradient descent \citep{mandic2004generalized,chen2023generalized,li2024problem,hazan2015beyond,fang2018spider}, gradient tracking \citep{nedic2017achieving,qu2017harnessing}, and multi-consensus with Chebyshev acceleration \citep{scaman2017optimal,hendrikx2021optimal,ye2023multi}.
Here the matrices $\mX^t$ and~$\mV^t$ consist of the local variables and local gradient estimators, i.e.,
\begin{align*}
\mX^t = \begin{bmatrix}
    \vx_1^t \\ \vdots \\ \vx_m^t
\end{bmatrix}\in\mathbb{R}^{m \times d}    
\qquad\text{and}\qquad
\mV^t = \begin{bmatrix}
    \vv_1^t \\ \vdots \\ \vv_m^t
\end{bmatrix}\in\mathbb{R}^{m \times d}.
\end{align*}

It is worth noting that the relaxed 
smoothness condition (Assumption \ref{asm:sym-1}) does not guarantee that the local gradient is Lipschitz continuous or bounded.
However, the existing analysis of decentralized methods mainly depends on the Lipschitz continuity or the boundedness of the local gradient \citep{song2024optimal,xin2020general,hendrikx2021optimal,li2022variance,metelev2024decentralized,liu2024decentralized,luo2022complexity,ye2023multi,bai2024complexity,yuan2022revisiting,lu2021optimal,lan2020communication,kovalev2024lower,lin2023decentralized,sahinoglu2024online,li2024problem}, which is not applicable to our setting.  

To address above issues, we design the Lyapunov function
\begin{align*}
\small\begin{split}    
\Phi^t=  f(\bx^t) + \frac{3\eta}{\sqrt{m}}\left(M_0 + M_1\Norm{\nabla f(\bx^t)}\right)\Norm{\mX^t-\vone\bx^t} + \frac{2\eta}{\sqrt{m}}\Norm{\mV^t-\vone\bv^t},
\end{split}
\end{align*}
where $M_0=\sqrt{2(L_0^2 +L_1^2\zeta^2)}$ and $M_1=\sqrt{2}L_1$.
In particular, our $\Phi^t$ includes the term 
\begin{align*}
    \left(M_0 + M_1\Norm{\nabla f(\bx^t)}\right)\Norm{\mX^t-\vone\bx^t}
\end{align*}
to align with the measurement unit on the right-hand side of inequality (\ref{eq:global-relax-smooth}) for the relaxed smoothness of the objective.
This design can well address the consensus errors $\Norm{\mX^t-\vone\bx^t}$ and $\Norm{\mV^t-\vone\bv^t}$ even if the local gradient is \textit{non-Lipschitz} and \textit{unbounded}. 

In contrast, the existing analysis essentially only considers the linear combinations of the function value, gradient norm, and the consensus errors \citep{song2024optimal,xin2020general,hendrikx2021optimal,li2022variance,metelev2024decentralized,liu2024decentralized,luo2022complexity,ye2023multi,bai2024complexity,yuan2022revisiting,lu2021optimal,lan2020communication,kovalev2024lower,lin2023decentralized,sahinoglu2024online,li2024problem}, so that
it does not work on the relaxed smoothness. 

By using the Lyapunov function $\Phi^t$, we can show the main theoretical results of proposed DNSGD (Algorithm \ref{alg:DNSGD}) as follows.

\begin{thm}\label{thm:main}
    Under Assumptions \ref{asm:lower}--\ref{ams:W}, running DNSGD (Algorithm~\ref{alg:DNSGD}) by taking 
    \begin{align*}
        & \eta = \min\left\{\frac{\epsilon}{4L_f+1}, \frac{1}{2L_1}\right\},~~ 
        b \geq \max\left\{\frac{256(4L_f+1)^2\sigma^2}{mL_f^2\epsilon^2},  \frac{1024L_1^2\sigma^2}{m L_f^2}\right\}
        \\
        & T \geq \max\left\{\frac{8(4L_f+1)\Delta_{\Phi}}{\epsilon^2}, \frac{16L_1\Delta_{\Phi}}{\epsilon}\right\},~~ 
        K = \fO\left(\frac{\log {(\sqrt{m}(1+\eta)})}{\sqrt{\gamma}}\right)
    \end{align*}
    guarantees $\BE\left[\Norm{\nabla f(\hat\vx)}\right] \leq \epsilon/2$, where $\hat\vx$ is uniformly sampled from~$\{\bx^0,\dots,\bx^{T-1}\}$, $L_f=L_0+L_1\zeta$, and
    \begin{align*} 
    \Delta_{\Phi} = f(\bx^0) - \inf_{\vx\in\BR^d} f(\vx) + \frac{2\eta}{\sqrt{m}}\BE\left[\Norm{\mV^0-\vone\bv^0}\right]. 
    \end{align*}
\end{thm}

\begin{remark}
In real-world applications, we typically do not require $\epsilon$ to be very small, so that the batch size $b=\fO(1/\epsilon^2)$ is not necessary to be very large.
Additionally, the batch size $b=\fO(1/\epsilon^2)$ is a standard setting in the analysis of stochastic distributed optimization, which is necessary to simultaneously achieve the tight sample complexity and the tight communication complexity in theoretical. 
Recall that even for the standard smoothness assumption ($L_1=0$), the lower bounds on the overall sample complexity and the overall communication complexity are $\fO(1/\epsilon^4)$ and $\fO(1/\epsilon^2)$, respectively \cite{lu2021optimal,yuan2016convergence}. 
Hence, after every communication round, we need at least the sample complexity of~$\fO(1/\epsilon^2)$ on average, matching the setting of batch size $\fO(1/\epsilon^2)$ in our theoretical analysis.    
\end{remark}

Note that the expected $(\epsilon/2)$-stationary point $\hat\vx$ in Theorem~\ref{thm:main} is based on the mean vectors~$\{\bx^t\}_{t=0}^{T-1}$, which are unavailable in decentralized setting. 
In addition, the expression of $\Delta_{\Phi}$ indicates the initial consensus error $\Norm{\mV^0-\vone\bv^0}$ may heavily affect the iteration number $T$ if it is significantly larger than the standard initial function value gap
$\Delta_f=f(\bx^0) - \inf_{\vx\in\BR^d} f(\vx)$
used in the analysis of centralized optimization.
Fortunately, we can address above issues by appropriate setting of $\hat K$ in line \ref{line:V0} to reduce the magnitude of $\Norm{\mV^0-\vone\bv^0}$ to guarantee
$\Delta_\Phi=\Theta(\Delta_f)$.
Finally, we achieve the refined upper complexity bounds as follows.
\begin{cor}\label{cor:main}
    Following the setting of Theorem \ref{thm:main}, we further set the initial communication rounds as $\hat K=\tilde\fO(1/\sqrt{\gamma})$, then it holds that~$\Delta_\Phi=\Theta(\Delta_f)$ and the output of DNSGD (Algorithm~\ref{alg:DNSGD}) guarantees that $\BE\left[\Norm{\nabla f(\hat\vx_i)}\right] \leq \epsilon$ for each agent $i\in[m]$, where $\hat\vx_i$ is uniformly sampled from $\{\vx_i^0,\dots,\vx_i^{T-1}\}$ generated by the $i$th agent.
    Additionally, each agent requires the sample complexity of
    \begin{align}\label{eq:upper-bound-sample}
    Tb = \fO\left(\frac{1}{m}\left(\frac{L_f\sigma^2\Delta_f}{\epsilon^4} + \frac{\sigma^2}{\epsilon^2} + \frac{L_1^3\sigma^2\Delta_f}{L_f^2 \epsilon} + \frac{L_1^2\sigma^2 }{L_f^2} \right)\right) 
    \end{align}
    and the algorithm requires the communication complexity of 
    \begin{align*}
        TK+\hat K = \tilde\fO\left(\frac{L_f\Delta_f}{\sqrt{\gamma}\epsilon^2} + \frac{L_1\Delta_f}{\sqrt{\gamma}\epsilon}\right),
    \end{align*}
    where $L_f=L_0+L_1\zeta$, $\Delta_f=f(\bx^0) - \inf_{\vx\in\BR^d} f(\vx)$, and we use the notation $\tilde\fO(\cdot)$ to hide the log factors.
\end{cor}

The term $1/m$ in the sample complexity bound (\ref{eq:upper-bound-sample}) implies our method achieves the linear speedup. 
Although expression (\ref{eq:upper-bound-sample}) looks somewhat complicated, it is easy to understand since we can verify our bounds match or surpass the best-known results in several special cases:
\begin{itemize}[topsep=-4pt,itemsep=1.5pt,partopsep=1.8pt, parsep=1.8pt,leftmargin=0.65cm]   
    \item In the case of $L_1=0$, the local functions reduce to the setting of standard $L_0$-smooth and it leads the objective $f$ is also $L_0$-smooth without Assumption \ref{asm:heterogeneous} and $L_f=L_0$. Then Corollary~\ref{cor:main} indicates the sample complexity 
    \begin{align*}
       Tb=\fO\left(\frac{1}{m}\left(\frac{L_0\sigma^2\Delta_f}{\epsilon^4} + \frac{\sigma^2}{\epsilon^2}\right)\right) 
    \end{align*}
    and the communication complexity of
    \begin{align*}
        TK+\hat K = \tilde\fO\left(\frac{L_0\Delta_f}{\sqrt{\gamma}\epsilon^2}\right),
    \end{align*}
    respectively. This matches the (near-)optimal results for decentralized stochastic nonconvex optimization under the standard smoothness \cite[Corollary~1]{lu2021optimal}.
    \item In the case of $m=1$, the problem reduces to the relaxed smooth optimization on single machine. This leads to $\zeta=0$ and $L_f=L_0$ since there is no heterogeneity.
    Then Theorem \ref{thm:main} indicates the sample complexity of
    \begin{align}\label{eq:sample-single}
        Tb = \fO\left(\frac{L_0\sigma^2\Delta_f}{\epsilon^4} + \frac{\sigma^2}{\epsilon^2} + \frac{L_1^3\sigma^2\Delta_f}{L_0^2 \epsilon} + \frac{L_1^2\sigma^2}{L_0^2} \right), 
    \end{align}
    which is tighter than the complexity of gradient clipping method in the non-distributed case \citep{zhang2020improved}. Recall Theorem 3.2 of \citet{zhang2020improved} claims the sample complexity of 
    \begin{align*}
    \fO(L_0\sigma^2\Delta_f\epsilon^{-4} + L_0^{-3}L_1^4\sigma^2\Delta_f).    
    \end{align*}
    It is worth noting that term $L_0^{-3}L_1^4\sigma^2\Delta_f$ dominates their bound only when $\epsilon=\Omega(L_0/L_1)$, and it is larger than the term $L_0^{-2}L_1^3\sigma^2\Delta_f\epsilon^{-1}$ in our (\ref{eq:sample-single}) in this case.
    In addition, \citet{zhang2020improved} assume the algorithm can access the exact gradient at the initial point, while their complexity bound does not include the cost of achieving the the exact gradient at the initial point. 
    In contrast, our analysis considers the access of the initial gradient estimator by stochastic first-order oracle, leading to the terms~$\sigma^2\epsilon^{-2}$ and~$L_1^2L_0^{-2}\sigma^2$ in our upper bound (\ref{eq:sample-single}).
    Recently, \citet{liu2024nonconvex} study the more general stochastic gradient estimator with heavy-tailed noise, while applying their result to our setting leads to the term $\sigma^3\epsilon^{-3}$ in the upper bound on the sample complexity \citep[\mbox{Theorem 3.2~with} $p=2$]{liu2024nonconvex}. 
    \item Following the case of $m=1$, we further consider the deterministic case, i.e., replace the mini-batch stochastic gradient in Algorithm \ref{alg:DNSGD} with the exact gradient and omit the communication steps.
    Then Theorem \ref{thm:main} indicates the iteration (gradient) complexity of  
    \begin{align*}
        T=\fO\left(\frac{L_0\Delta_f}{\epsilon^2}+\frac{L_1\Delta_f}{\epsilon}\right), 
    \end{align*}
    matching the best-known result for solving relaxed smooth problem by exact gradient method on the single machine \cite[Theorem 3.1]{vankov2024optimizing}.

\end{itemize}

  
\section{The Complexity Analysis}

In this section, we briefly sketch the derivations of our main results. The details for all proofs are presented in the appendices.
We first present the following descent lemma with respect to the mean vectors $\bx^{t+1}=\frac{1}{m}\sum_{i=1}^m\vx_i^{t+1}$ and $\bx^t=\frac{1}{m}\sum_{i=1}^m\vx_i^t$.

\begin{lem}\label{lem:descent-mean-2}
Under the setting of Theorem \ref{thm:main}, we have
\begin{align*}    
\BE[f(\bx^{t+1})]   
\leq & f(\bx^t) - \eta\left(1-\frac{\eta L_1}{2}\right)\Norm{\nabla f(\bx^t)} \\
& + \frac{2\sqrt{2}\eta}{\sqrt{m}}\left(M_0+M_1\Norm{\nabla f(\bx^t)}\right)\Norm{\mX^t-\vone\bx^t} \\
& + \frac{\eta}{\sqrt{m}}\Norm{\mV^t - \vone\bv^t} + \frac{2\eta\sigma}{\sqrt{mb}} + \frac{\eta^2L_f}{2}
\end{align*}
for given $\mX^t$ and $\mV^t$.
\end{lem}  

In the remains of this paper, we define 
\begin{align*}
\rho = c_1\big(1 - c_2\sqrt{1 - \lambda_2(\mW)}\,\big)^K    
\end{align*}
with $c_1=\sqrt{14}$ and $c_2=1-{1}/{\sqrt{2}}$, then
Proposition~\ref{proposition_fastmix} implies we can guarantee $\rho$ be sufficient small by taking appropriate communication rounds number $K$.
We then establish the recursions for the consensus errors as follows.

\begin{lem}\label{lem:recursion-X}
For Algorithm \ref{alg:DNSGD}, we have 
\begin{align*}
\Norm{\mX^{t+1}-\vone\bx^{t+1}} 
\leq \rho\left(\Norm{\mX^t - \vone\bx^t} + m\eta\right) 
\end{align*}
and
\begin{align*}
\Norm{\mX^t-\vone\bx^t} \leq \frac{\rho m\eta}{1-\rho}.
\end{align*}
\end{lem}

\begin{lem}\label{lem:recursion-V}
We take 
\begin{align*}
K\geq \frac{2+\sqrt{2}}{\sqrt{\gamma}}\log \left(\sqrt{14}(1+m \eta L_1)\right)    
\end{align*}
for Algorithm \ref{alg:DNSGD}, then it holds
\begin{align*}
   & \BE\left[\Norm{\mV^{t+1} - \vone\bv^{t+1}}\right] \\
\leq & \rho\Bigg(\Norm{\mV^t - \vone\bv^t} + \left(M_2 + M_3\Norm{\nabla f(\bx^{t})}\right)\Norm{\mX^t - \vone\bx^t} \\
& \quad~ + \left(M_2 + M_3\Norm{\nabla f(\bx^{t})}\right) m\eta + \frac{2\sqrt{m}\sigma}{\sqrt{b}}\Bigg),
\end{align*}
for given $\mX^t$ and $\mV^t$, 
where 
\begin{align*}
M_2 = (\rho+1)L_f+\rho L_fL_1\eta
\quad\text{and}\quad 
M_3 = \rho L_1(L_1\eta+1)+L_1.    
\end{align*}
\end{lem}

Applying Lemmas \ref{lem:descent-mean-2}, \ref{lem:recursion-X}, and \ref{lem:recursion-V}, we can characterize the decease of the Lyapunov function as 
\begin{align}\label{eq:Phi-decease}    
 \BE[\Phi^{t+1}]
\leq  \Phi^t - \frac{5\eta}{8}\Norm{\nabla f(\bx^t)} + \frac{3\eta^2L_f}{4}.
\end{align}
Summing equation (\ref{eq:Phi-decease}) over $t=0,\dots,T-1$, 
we can achieve our main result in Theorem~\ref{thm:main} that 
guarantees $\BE\left[\Norm{\nabla f(\hat\vx)}\right] \leq \epsilon/2$, where $\hat\vx$ is uniformly sampled from $\{\bx^1,\dots,\bx^m\}$. 

In Theorem \ref{thm:main}, the setting of iterations number $T$
depends on the term $\Delta_{\Phi}$, which may be dominated by the expected consensus error $\BE\left[\Norm{\mV^0-\vone\bv^0}\right]$.
However, the results in the single machine scenario only depend on $\Delta_f=f(\bx^0) - \inf_{\vx\in\BR^d}f(\vx)$.
To fill this gap, we set the initial communication rounds as $\hat K=\tilde\fO(1/\sqrt{\gamma})$ for line~\ref{line:V0} of Algorithm \ref{alg:DNSGD}, then Proposition \ref{proposition_fastmix} guarantees 
\begin{align*}
\frac{2\eta}{\sqrt{m}}\BE\left[\Norm{\mV^0-\vone\bv^0}\right]\leq \fO(\Delta_f)     
\end{align*}
to achieve $\Delta_{\Phi}\leq \fO(\Delta_f)$.
Therefore, the upper bound of the iteration number $T$ shown Theorem \ref{thm:main} to guarantee $\BE\left[\Norm{\nabla f(\hat\vx)}\right] \leq \epsilon/2$ can be improved to 
$\fO(L_f\Delta_f\epsilon^{-2}+L_1\Delta_f\epsilon^{-1})$.

For achieving an $\epsilon$-stationary point on each node, 
we consider the gradient at $\hat\vx_i$ as follows
\begin{align*}
 \BE[\Norm{\nabla f(\hat\vx_i)}] 
=&  \BE\left[\frac{1}{T}\sum_{t=0}^{T-1}\Norm{\nabla f(\vx_i^t)}\right]  \\
\leq & \BE\bigg[\underbrace{\frac{1}{T}\sum_{t=0}^{T-1}\Norm{\nabla f(\vx_i^t)-\nabla f(\bx^t)}}_{A_1} + \underbrace{\frac{1}{T}\sum_{t=0}^{T-1}\Norm{\nabla f(\bx^t)}}_{A_2}\bigg].
\end{align*}
For term $A_1$, Assumption \ref{asm:sym-1} implies
\begin{align*}
\BE[A_1] 
\leq & \BE\left[\frac{1}{T}\sum_{t=0}^{T-1}\big(L_f+L_1\Norm{\nabla f(\bx^t)}\big)\Norm{\mX^t - \vone\bx^t
}\right]  \\
\leq & \BE\left[\frac{1}{T}\sum_{t=0}^{T-1}\frac{\rho m\eta\big(L_f+L_1TB\big)}{1-\rho}\right],
\end{align*}
where the last step is based on 
Lemma \ref{lem:diff-bar-x} in Appendix \ref{sec:desent-proof}. 
Recall that Theorem \ref{thm:main} shows $\BE[A_2]=\BE[\Norm{\nabla f(\hat\vx)}]\leq\epsilon/2$.
In addition, we can set the communication rounds to be $K=\tilde\fO(\sqrt{1/\gamma}\,)$ for the multi-consensus step  to make $\rho$ be sufficient small to guarantee that $\BE[A_1]\leq\epsilon/2$.
Therefore, we achieve the desired $\epsilon$-stationary point $\hat\vx_i$ at each agent $i$. 

\begin{remark}
Under Assumption \ref{asm:heterogeneous}, the algorithm without gradient tracking can also guarantee the convergence, but its communication complexity will include additional dependency on $\log\zeta$ even for the standard smoothness setting, where $\zeta$ is the upper bound of gradient  dissimilarity \citep[Section 5]{yuan2022revisiting}. 
In other words, employing the gradient tracking is beneficial to the communication efficiency for large $\zeta$.
\end{remark}

\section{Discussion and Related Work}

The design of our DNSGD (Algorithm \ref{alg:DNSGD}) is based on normalized (stochastic) gradient descent, which is extensively studied in centralized scenario \citep{mandic2004generalized,chen2023generalized,cutkosky2020momentum,hazan2015beyond,fang2018spider}. 
In a related work, \cite{li2024problem} studied normalized stochastic gradient descent in decentralized setting, while their analysis is only limited to the setting of standard smoothness.
In a very recent work, \citet[Algorithm 2 and Theorem 4]{jiang2025decentralized} claims a first-order method for decentralized optimization under the relaxed smoothness, while each of their iteration requires every node to perform the step of gradient clipping which depends on accessing the mean vector $\bx_i^t$ and the corresponding exact local gradient~$\nabla f_i(\bx^t)$.
Recall that the decentralized method typically only allows each node directly communicates with its neighbors (Assumption \ref{ams:W}) and the stochastic first-order oracle can only access the unbiased gradient estimator with bounded variance (Assumption~\ref{asm:sfo}).
Therefore, we argue that their algorithm is not truly decentralized or stochastic.
In contrast, our method only requires each node to iterate with its local stochastic gradient (line \ref{line:Gt} of Algorithm \ref{alg:DNSGD}), which is fully decentralized and stochastic.

The convergence analysis for our DNSGD (Algorithm \ref{alg:DNSGD}) depends on the assumption of bounded gradient dissimilarity (Assumption \ref{asm:heterogeneous}), which is commonly used in distributed optimization \citep{woodworth2020minibatch,gorbunov2021local,assran2019stochastic,lian2017can}.
Specifically, we use this assumption to establish the inequality 
$\Norm{\nabla f(\bx^{t+1})} \leq (L_0+L_1\zeta)\eta + \left(1+L_1\eta\right)\Norm{\nabla f(\bx^{t})}$,
which connects the gradients at the points $\bx^{t}$ and $\bx^{t+1}$ 
(see Lemma \ref{lem:recursion-grad} in Appendix~\ref{sec:desent-proof}).
We also combine Assumption \ref{asm:heterogeneous} with the $(L_0,L_1)$-smoothness of local functions (Assumption \ref{asm:sym-1}) to guarantee the global objective $f=\frac{1}{m}\sum_{i=1}^m f_i$ is $(L_0+L_1\zeta,L_1)$-smooth 
(see the proof in Appendix~\ref{sec:proof-global-l0l1}).
However, how to establish the convergence for decentralized relaxed smooth optimization without bounded gradient dissimilarity (Assumption~\ref{asm:heterogeneous}) is still an open problem.
It is worth noting that simply assuming that each local function $f_i$ is $(L_0, L_1)$-smooth even cannot guarantee that the global objective function $f=\frac{1}{m}\sum_{i=1}^m f_i$ is also $(L_0, L_1)$-smooth, which is different with the standard smooth case.
We illustrate this point with the following example.

\begin{prop}\label{prop:example}
We consider functions
\begin{align*}
    f_1(x) = \exp(Lx),~~f_2(x) = \exp(-Lx),~~\text{and}~~f(x) = \frac{f_1(x)+f_2(x)}{2}
\end{align*}
for some $L>0$ and let $L_1=L/\log 2$, then both functions $f_1$ and $f_2$ is $(0,L_1)$-smooth while the function $f$ is not $(0,L_1)$-smooth.
\end{prop}

\begin{proof}
We can verify that the derivatives of functions $f_1$, $f_2$, and~$f$ have the form of $f'_1(x) = L\exp(Lx)$, $f'_2(x) = -L\exp(-Lx)$, and $f'(x) = (L\exp(Lx)-L\exp(-Lx))/{2}$.
Note that the function
\begin{align*}
    h(z) = \frac{|1-\exp(Lz)|}{L|z|} 
\end{align*}
is increasing.
Taking $z=y-x\neq 0$, we have
\begin{align}\label{eq:example-h}
\begin{split}    
  \frac{|1-\exp(L(y-x))|}{L|y-x|} 
= h(y-x)  
\leq   h\left(\frac{1}{L_1}\right)  
= \frac{|1-\exp(L/L_1)|}{L/L_1}
\end{split}
\end{align}
for all $x,y\in\BR$ with $|y-x|\in[-1/L_1, 1/L_1]$.

For the function $f_1$, we have
\begin{align*}
  &  |f'_1(x) - f'_1(y)|  
=   |L\exp(Lx) - L\exp(Ly)| \\
= & L\exp(Lx)|1-\exp(L(y-x))|  \\
\leq &  L\exp(Lx) \cdot L_1|1-\exp(L/L_1)||y-x|  \\
= & L \exp(Lx) |y-x| L_1 
=   L_1 f_1'(x)|x-y|,
\end{align*}
where the inequality is based on equation (\ref{eq:example-h}) and the last two steps are based on $f'_1(x) = L\exp(Lx)$ and the setting $L_1=L/\log 2$.
Therefore, we have shown that $f_1$ is $(0,L_1)$-smooth. 
Similarly, we can also show that $f_2$ is $(0,L_1)$-smooth. 

Let $x=0$ and $y=(\log 2)/L$, which holds $|x-y|=1/L_1$.
Note that we also have
\begin{align*}
    |f'(x) - f'(y)|  
=& \frac{L}{2}\left|\exp(Lx)\!-\!\exp(-Lx) -(\exp(Ly)\!-\!\exp(-Ly))\right| \\
=& \frac{L}{2}\left|\exp(0)-\exp(0) -\left(2-\frac{1}{2}\right)\right| = \frac{3L}{4}
\end{align*}
and
\begin{align*}
    L_1f'(x)|y-x| = \frac{L_1L(\exp(Lx)-\exp(-Lx))|y-x|}{2} = 0.
\end{align*}
Therefore, the inequality $|f'(x) - f'(y)| \leq L_1 f'(x)|y-x|$ does not hold, which means the function~$f$ is not $(0,L_1)$-smooth.
\end{proof}



\begin{figure*}[ht]
    \centering
    \begin{tabular}{ccc}
    \includegraphics[scale=0.42]{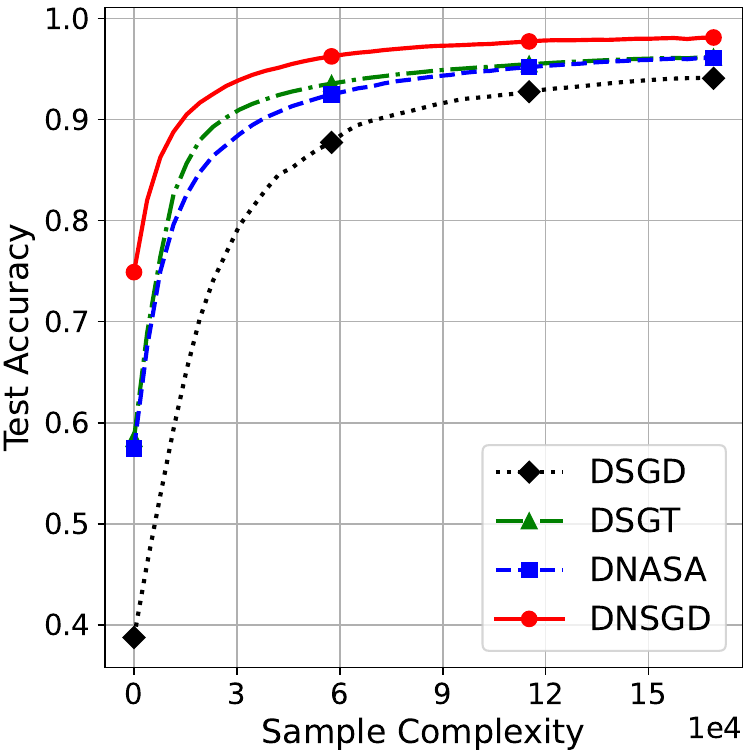} &
    ~\includegraphics[scale=0.42]{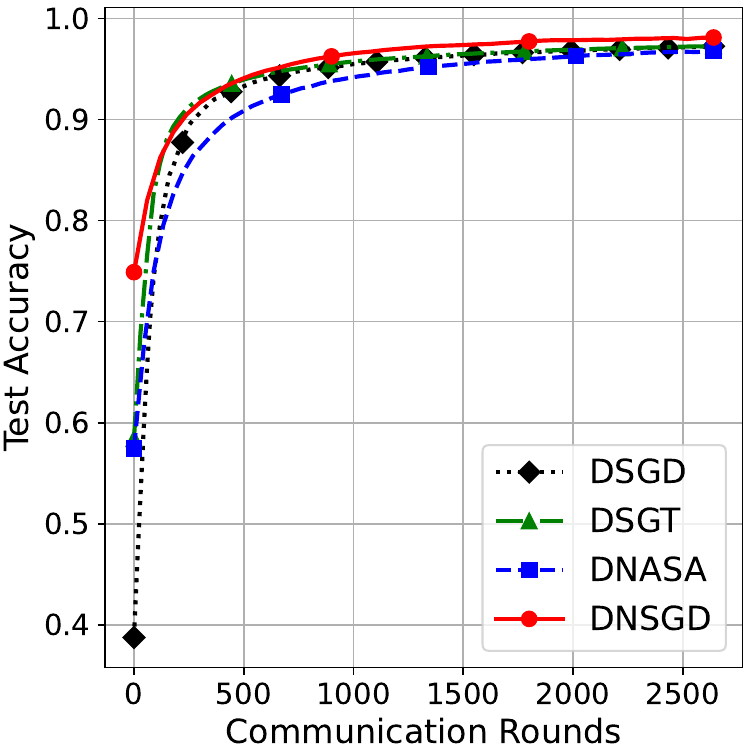}~ & \includegraphics[scale=0.42]{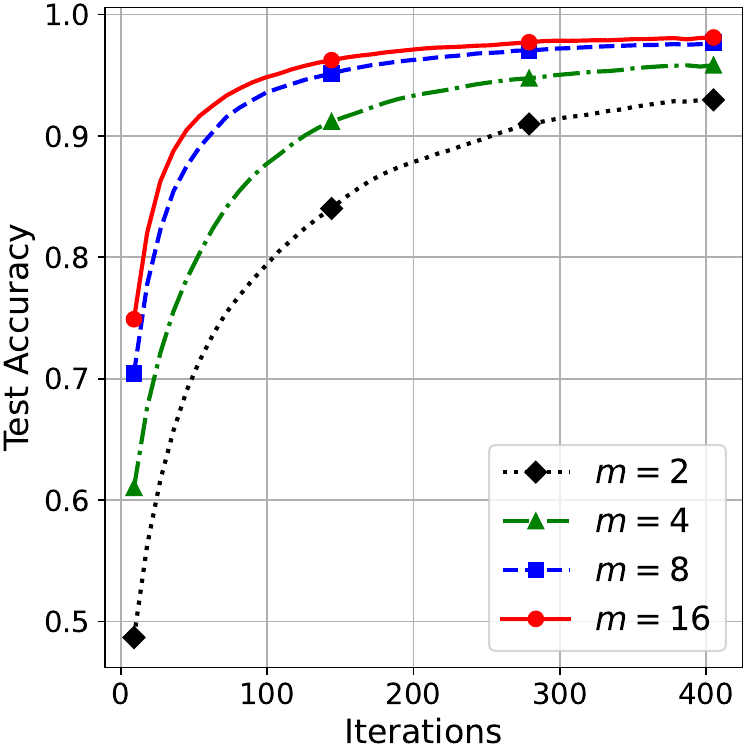} \\
     (a) sample vs. accuracy &
     (b) communication vs. accuracy &
     (c) linear speed-up
    \end{tabular} \vskip-0.1cm
    \caption{We present numerical results on the MNIST dataset over the  Erd{\"o}s--R{\'e}nyi network. The subfigures (a) and (b) present the sample complexity per agent and the communication rounds against the average test accuracy on the network with $m=16$ agents. 
    The subfigure (c) verifies the linear speedup effect of our DNSGD on the network with $m=2,4,8,16$ agents.}
    \label{fig:mnist-renyi} 
\end{figure*}

\begin{figure*}[ht]
    \centering
    \begin{tabular}{ccc}
    \includegraphics[scale=0.42]{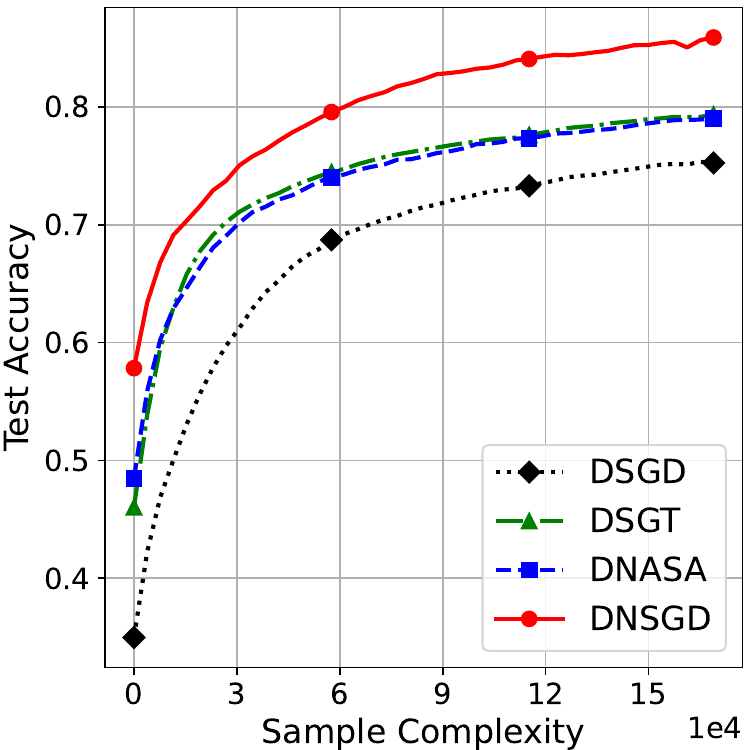} &
    ~\includegraphics[scale=0.42]{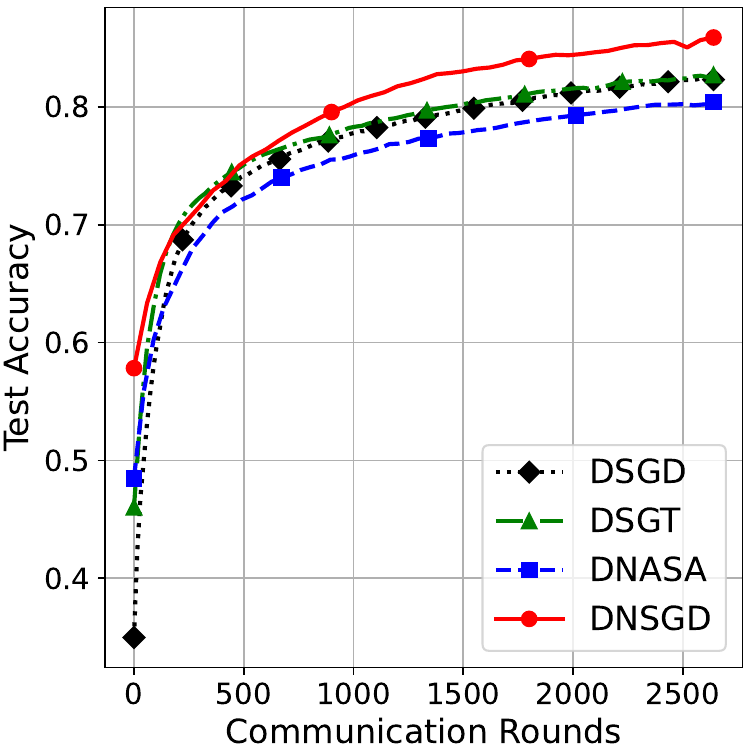}~ & \includegraphics[scale=0.42]{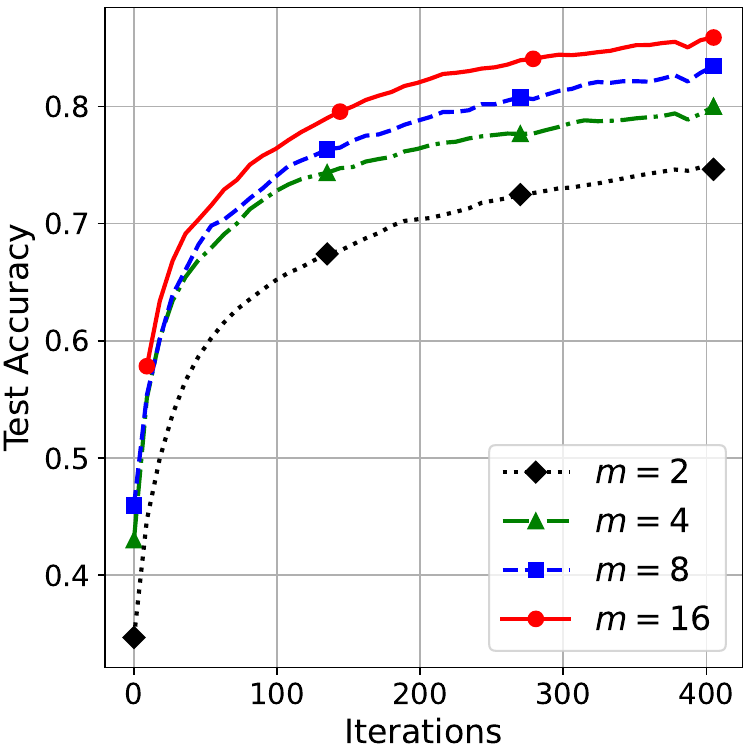} \\
     (a) sample vs. accuracy &
     (b) communication vs. accuracy &
     (c) linear speed-up 
    \end{tabular} \vskip-0.1cm
    \caption{We present numerical results on the Fashion-MNIST dataset over the  Erd{\"o}s-R{\'e}nyi network. The subfigures (a) and (b) present the sample complexity per agent and the communication rounds against the average test accuracy on the network with $m=16$ agents. 
    The subfigure (c) verifies the linear speedup effect of our DNSGD on the network with $m=2,4,8,16$ agents.}
    \label{fig:fashion-renyi}
\end{figure*}

Recall that the discussion after Corollary \ref{cor:main} shows that the sample complexity of our method matches the lower bound in the special case of $L_1=0$ \citep{lu2021optimal}. 
This implies that the $\fO(\epsilon^{-4})$ dependence on accuracy we obtained cannot be improved in our setting.
If we target to achieve the sharper~$\mathcal{O}(\epsilon^{-3})$ dependence in the sample complexity, the assumption of expected relaxed smoothness is necessary \citep{chen2023generalized,reisizadeh2025variance}, e.g., strengthen 
condition (\ref{eq:asm:sym-1}) in Assumption~\ref{asm:sym-1} to
$\BE[\Norm{\nabla F_i(\vx;\vxi_i)-\nabla F_i(\vy;\vxi_i)}^2] \leq \left(L_0 + L_1 \Norm{\nabla f_i(\vx)}\right)^2\Norm{\vx-\vy}^2$,
then applying the variance-reduced gradient estimators \citep{chen2023generalized,reisizadeh2025variance,fang2018spider} at local agents is possible to achieve the sample complexity depends on $\mathcal{O}(\epsilon^{-3})$.

\section{Experiments} \label{sec:exp}

We conduct the experiments on the applications of image classification and fine-tuning the large language model, which are performed on a cluster with 8 virtual GPUs, 8~virtual CPUs, and 96 GB memory.  
We compare our decentralized normalized stochastic gradient descent (DNSGD) with  decentralized stochastic gradient descent (D-SGD) \citep{lian2017can}, decentralized stochastic gradient descent with gradient tracking (D-SGT) \citep{nedic2009distributed,qu2017harnessing} , and decentralized normalized averaged stochastic approximation (D-NASA) \citep{li2024problem}.
The code is available at the link: 
\url{https://anonymous.4open.science/r/DNSGD-F8D6} 

\subsection{Image Classification} \label{sec:exp-image}

We present numerical results of image classification on datasets MNIST and Fashion-MNIST. The MNIST dataset is a well-known collection of handwritten digits (10 classes), containing $70,000$ images, with $60,000$ for training and $10,000$ for testing \citep{lecun2002gradient}. The Fashion-MNIST dataset serves as a more challenging variant, consisting of $70,000$ images of fashion items with 10 classes, such as trouser and dress \citep{xiao2017fashion}. Each example of MNIST or Fashion-MNIST is a 28$\times$28 grayscale image.

\begin{figure*}[t]
    \centering
    \begin{tabular}{ccc}
    \includegraphics[scale=0.42]{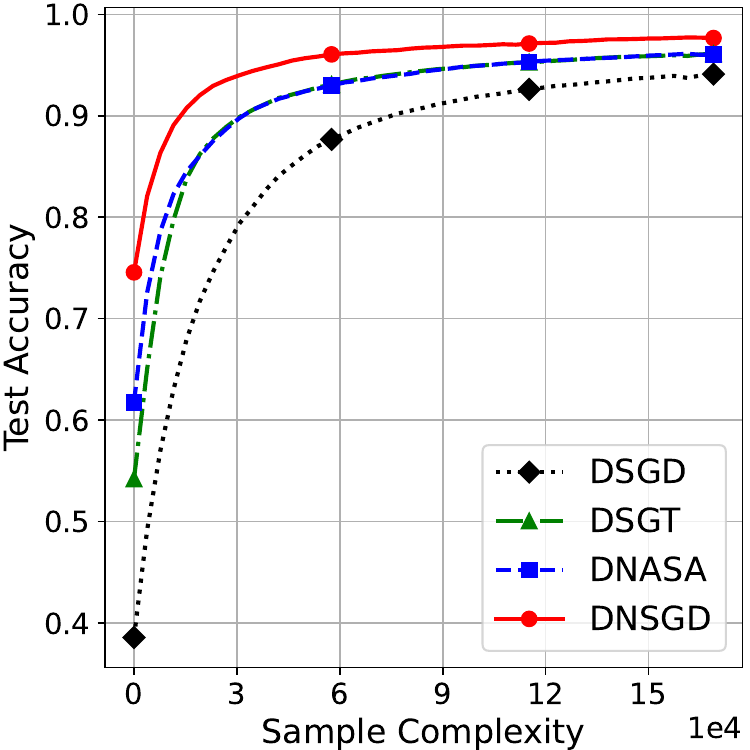} &
    ~\includegraphics[scale=0.42]{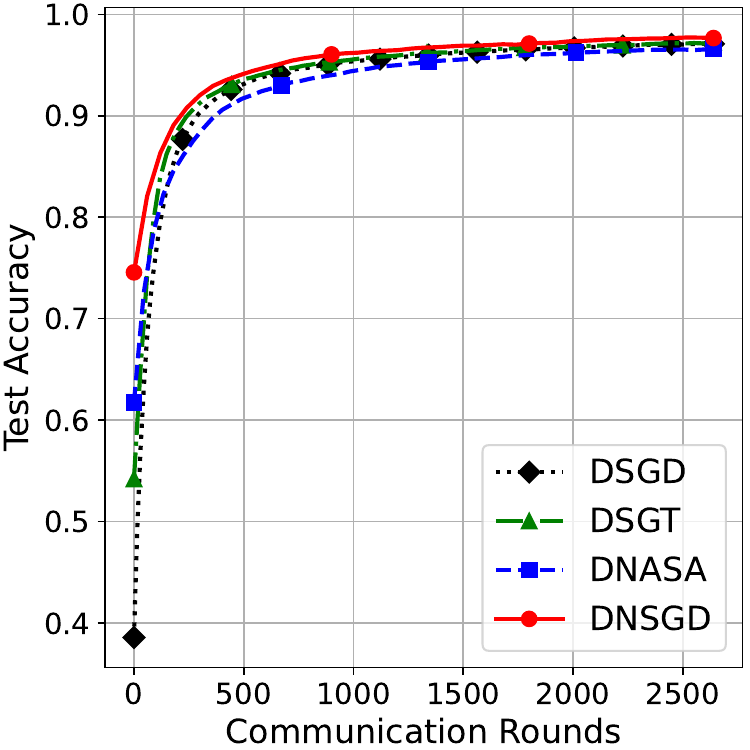}~ & \includegraphics[scale=0.42]{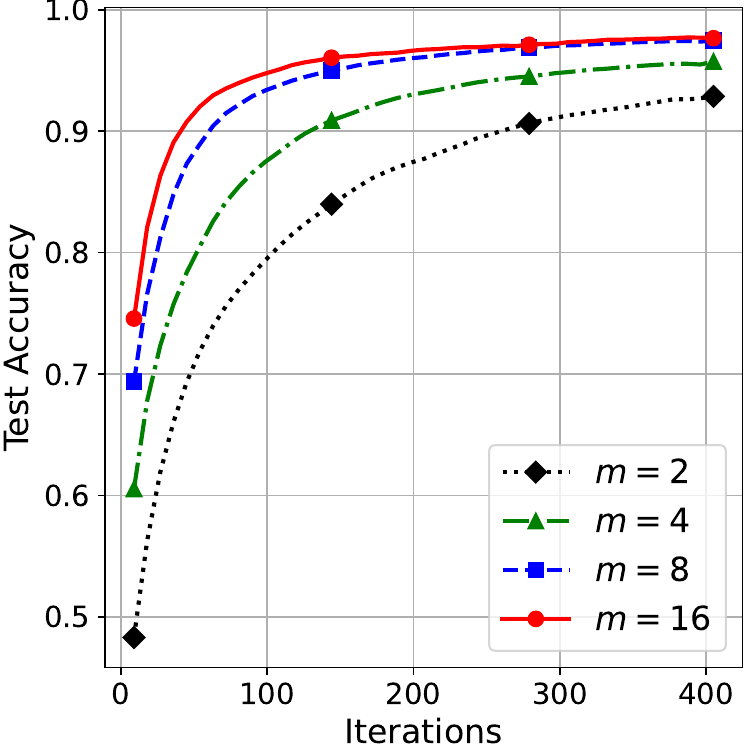} \\
     (a) sample vs. accuracy &
     (b) communication vs. accuracy &
     (c) linear speed-up 
    \end{tabular} \vskip-0.2cm
    \caption{We present numerical results on the MNIST dataset over the  ring network. The subfigures (a) and (b) present the sample complexity per agent and the communication rounds against the average test accuracy on the network with $m=16$ agents. 
    The subfigure (c) verifies the linear speedup effect of our DNSGD on the network with $m=2,4,8,16$ agents.}
    \label{fig:mnist-ring} \vskip-0.15cm
\end{figure*}

\begin{figure*}[t]
    \centering
    \begin{tabular}{ccc}
    \includegraphics[scale=0.42]{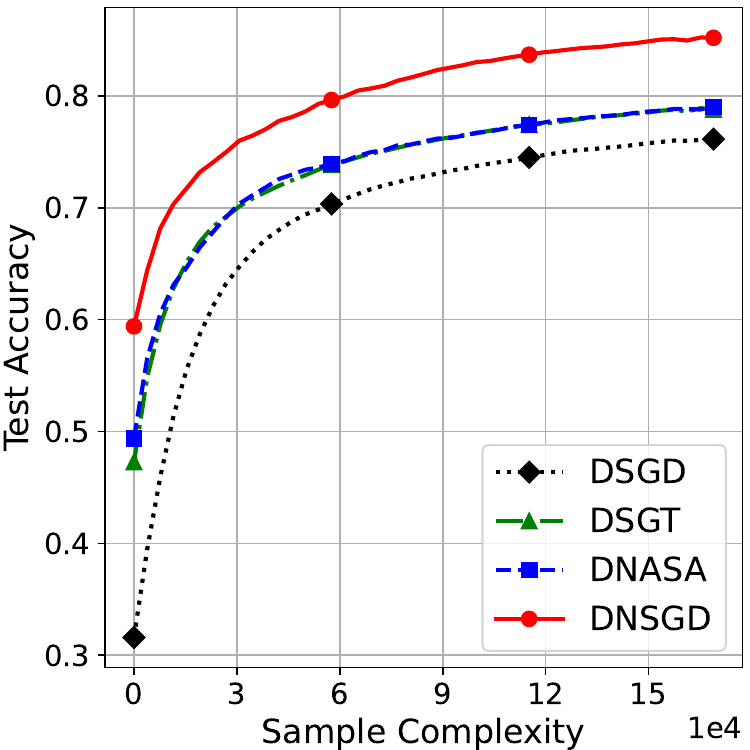} &
    ~\includegraphics[scale=0.42]{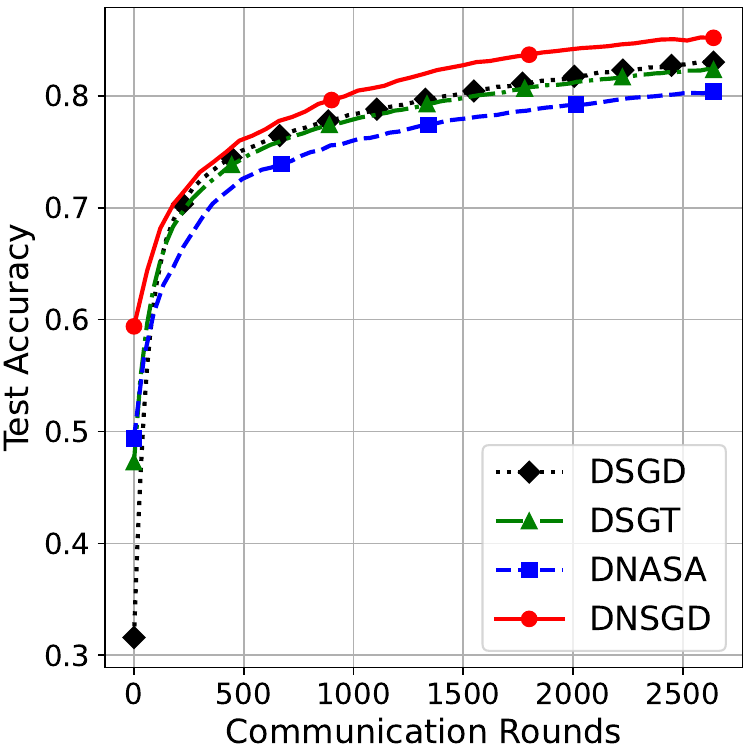}~ & \includegraphics[scale=0.42]{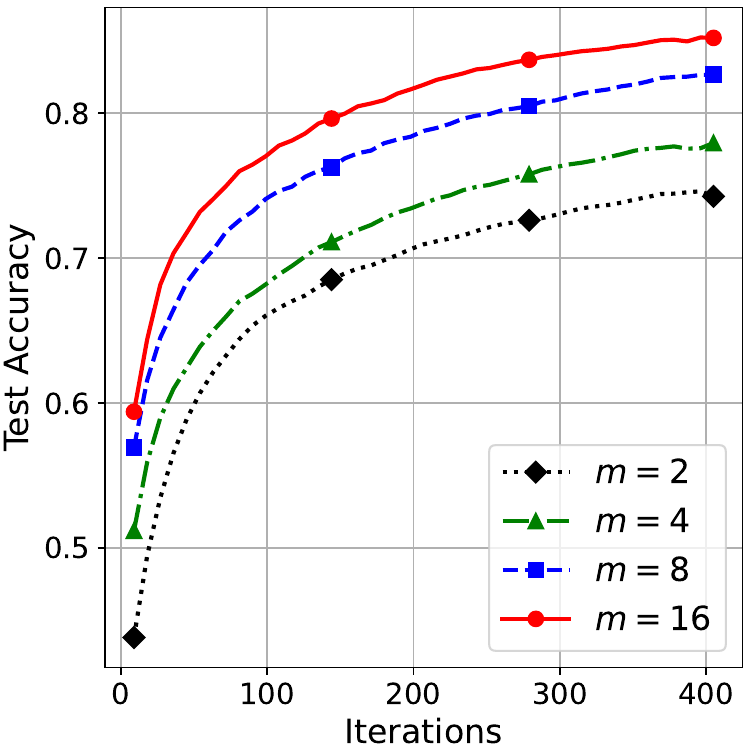} \\
     (a) sample vs. accuracy &
     (b) communication vs. accuracy &
     (c) linear speed-up 
    \end{tabular} \vskip-0.2cm
    \caption{We present numerical results on the Fashion-MNIST dataset over the ring network. The subfigures (a) and (b) present the sample complexity per agent and the communication rounds against the average test accuracy on the network with $m=16$ agents. 
    The subfigure (c) verifies the linear speedup effect of our DNSGD on the network with $m=2,4,8,16$ agents.}
    \label{fig:fashion-ring} \vskip-0.2cm
\end{figure*}

Our model employs a four-layer neural network, consisting of two convolutional layers and two fully connected layers. The first convolutional layer processes the single-channel input image with $32$ kernels of size $3 \times 3$, followed by ReLU activation and $2 \times 2$ max-pooling. The second convolutional layer takes the $32$-channel feature maps from the first layer, applies another $32$ kernels of size $3 \times 3$, and is followed by ReLU activation and $2 \times 2$ max-pooling. To enhance generalization, we apply dropout with a rate of $0.25$ after the convolutional layers. The resulting feature maps are then flattened and passed through a fully connected layer with $128$ units and ReLU activation, followed by a second dropout layer with a rate of $0.5$. Finally, we use a fully connected output layer to generate scores for the $10$ classes.


We perform numerical comparisons on a distributed system with $16$ agents, where the data is evenly distributed among them. 
In our experiments, we employ two network topologies, i.e., the Erd{\"o}s--R{\'e}nyi graph~\citep{erdds1959random} with the connectivity probability of $p = 0.6$ and the ring graph. 
For all algorithms, we set batch size to be $b = 250$. 
The step size of D-SGD, D-SGT and DNSGD are chosen from $\{0.0001,0.0005,0.001,0.005,0.01,0.05,0.1,0.5\}$. The learning rate of the D-NASA algorithm is given by $\eta_t=m^{1/4}t^{3/4}$ which follows the theory and implementation of \citet{li2024problem}, where $t$ is the iteration counter and $m$ is the number of agents. 
For our DNSGD, we set the iterations number of Chebyshev acceleration (Algorithm~\ref{alg:fm}) to~$\hat K = K = 2$.

We present the numerical results on the  Erd{\"o}s-R{\'e}nyi graph in Figures \ref{fig:mnist-renyi} and \ref{fig:fashion-renyi}, and the results on the ring graph are shown in Figures \ref{fig:mnist-ring} and~\ref{fig:fashion-ring}. 
We evaluation the performance by showing the sample complexity and the communication rounds against the average test accuracy, i.e., the mean of the test accuracies achieved by the models from all agents.   
We can observe that DNSGD outperforms all baseline methods. 
The linear speed-up of our method is verified by comparing the results on the networks with $m=2,4,8,16$ agents.
Figures \ref{fig:mnist-renyi}(c), \ref{fig:fashion-renyi}(c), \ref{fig:mnist-ring}(c), and \ref{fig:fashion-ring}(c) show that the more agents do lead to the faster convergence.

\begin{figure*}[t]
    \centering
    \begin{tabular}{cc}
    \includegraphics[scale=0.5]{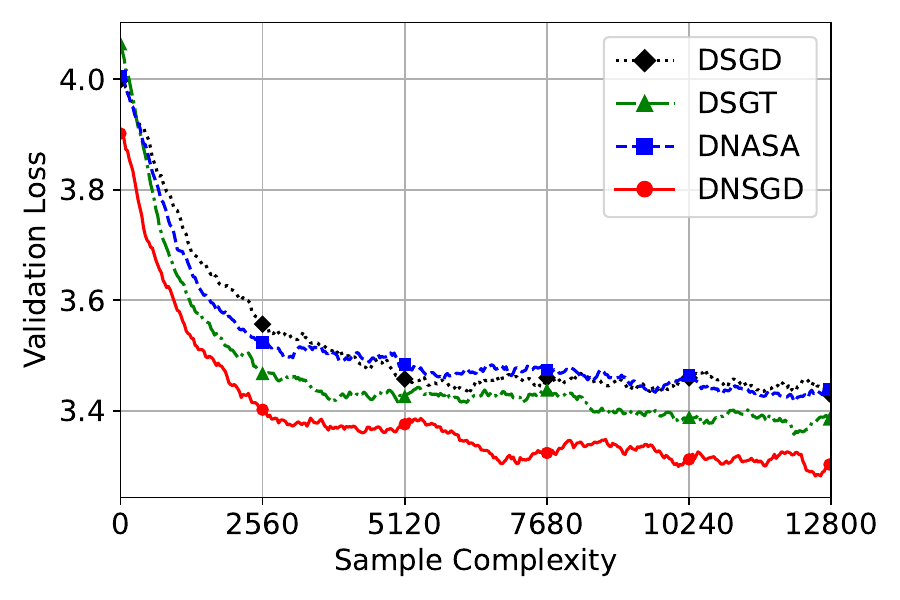} &
    ~\includegraphics[scale=0.5]{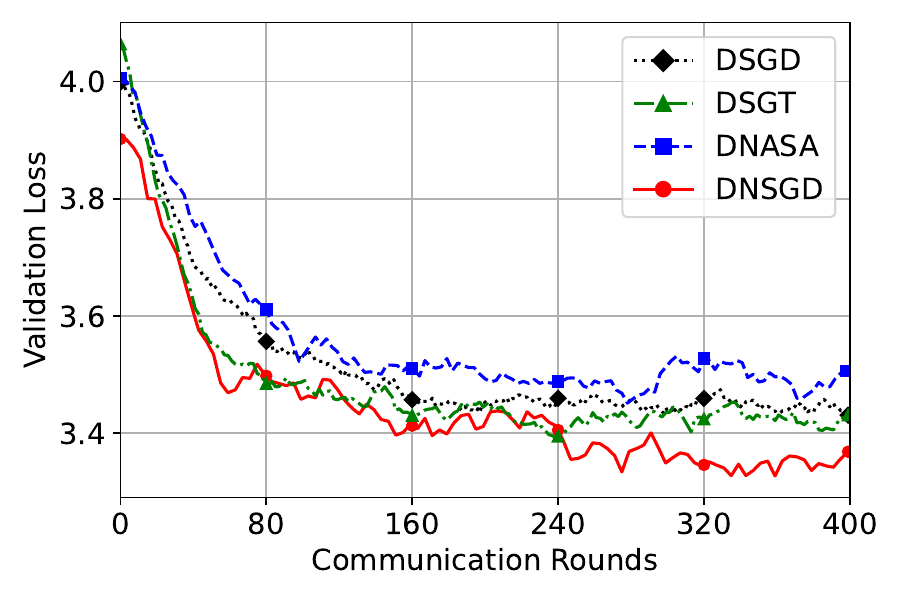}~ \\
      (a) sample vs. validation loss &
      (b) communication vs. validation loss
    \end{tabular} \vskip-0.15cm
    \caption{We present numerical results for fine-tuning nanoGPT on the shakespeare dataset over the Erd{\"o}s-R{\'e}nyi network with $m=4$ agents. The subfigures (a) and (b) present the sample complexity per agent and the communication rounds against the validation loss.}
    \label{fig:shakespeare-renyi}  \vskip-0.15cm
\end{figure*}

\begin{figure*}[t]
    \centering
    \begin{tabular}{cc}
    \includegraphics[scale=0.5]{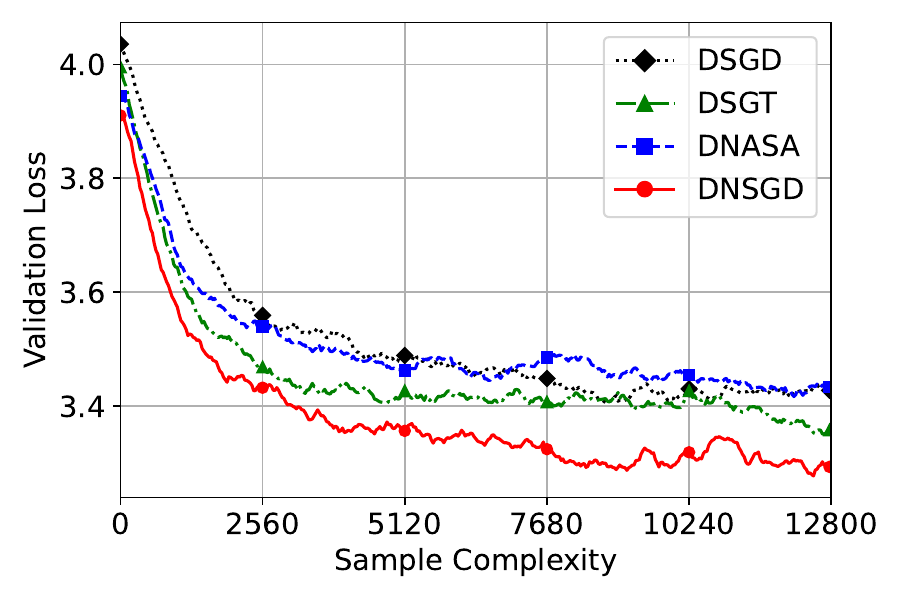} &
    ~\includegraphics[scale=0.5]{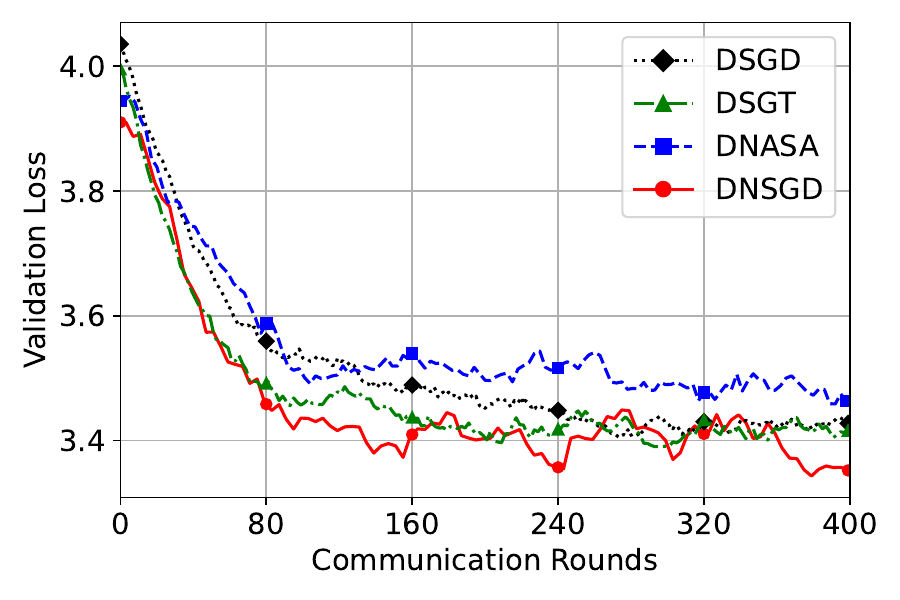}~ \\
      (a) sample vs. validation loss &
      (b) communication vs. validation loss
    \end{tabular} \vskip-0.15cm
    \caption{We present numerical results for fine-tuning nanoGPT on the shakespeare dataset over the ring network with $m=4$ agents. The subfigures (a) and (b) present the sample complexity per agent and the communication rounds against the validation loss.}
    \label{fig:shakespeare-ring}  \vskip-0.15cm
\end{figure*}

\subsection{Fine-Tuning the Large
Language Model}
We provide numerical results for
fine-tuning large language models. Specifically, we fine-tune nanoGPT ($124$M parameter)\footnote{https://github.com/karpathy/nanoGPT} on the Shakespeare dataset \cite{caldas2018leaf}\footnote{https://huggingface.co/datasets/flwrlabs/shakespeare}.

Similar to the empirical studies in Section~\ref{sec:exp-image},
we perform our experiments on two network topologies with 4 agents, i.e., the Erd{\"o}s--R{\'e}nyi graph~\citep{erdds1959random} with the connectivity probability of $p = 0.6$ and the ring graph. 
For all algorithms, we update the parameters of the model every $4$ micro-steps with batch-size $b=2$. 
The step sizes of D-SGD, D-SGT and DNSGD are chosen from $\{0.0001, 0.0005, 0.001, 0.005, 0.01, 0.05, 0.1, 0.5\}$. 
The step sizes of the D-NASA algorithm is given by $\eta_t = m^{1/4}t^{3/4}$ which follows the theory and implementation of \citet{li2024problem}, where $t$ is the iteration counter and $m$ is the number of agents. 
For our DNSGD, we set the iterations number of Chebyshev acceleration (Algorithm \ref{alg:fm}) to be~$\hat K = K = 2$.

We present the numerical results on the  Erd{\"o}s--R{\'e}nyi graph in Figure \ref{fig:shakespeare-renyi}, and the results on the ring graph are shown in Figure~\ref{fig:shakespeare-ring}. 
We evaluate the performance of the algorithms by showing the sample complexity and the communication rounds against the loss on the validation set. 
We can observe that our DNSGD outperforms all baseline methods with respect to both the measures of the computation and the computation.

\section{Conclusion}

This paper proposes decentralized normalized stochastic gradient
descent for nonconvex optimization under the relaxed smoothness. 
We provide theoretical analysis to show our method can achieve the $\epsilon$-stationary point on each agent.
In particular, we design a new Lyapunov function including the product of the gradient norm and the consensus error for convergence analysis, which successfully address the non-Lipschitz and unbounded local gradients in relaxed smooth decentralized optimization.
The empirical studies also show that our method performs better than baselines.

In future work, we would like to establish decentralized methods for convex optimization under the relaxed smoothness \citep{chezhegov2025convergence,gaash2025convergence,gorbunov2024methods,koloskova2023revisiting,lobanov2024linear,vankov2024optimizing} and study the possibility to use the momentum to improve the performance of the algorithm in practice \cite{cutkosky2020momentum}.
It is also interesting to extend our results to the more general settings such as the symmetric relaxed smoothness \citep{chen2023generalized} and the stochastic gradient estimator with heavy-tailed noise \citep{liu2024nonconvex}.

\begin{acks}
We would like to thank Yuxin Liu for the helpful discussion. 
This work is supported by the Major Key Project of Pengcheng Laboratory (No. PCL2024A06), the National Natural Science Foundation of China (No.12571557), the National Natural Science Foundation of China (No.62306116), Shanghai Basic Research Program (23JC1401000), and the Shanghai Special Fund for Promoting High-Quality Industrial Development (No. 20250307).
The computations in this research were performed using the CFFF platform of Fudan University.
\end{acks}

\bibliographystyle{ACM-Reference-Format}
\bibliography{sample-base}

\clearpage
\appendix

\section{Some Technical Lemmas}

We first provide some technical lemmas for later proofs.

\begin{lem}[{\citet[Lemma 12]{ye2023multi}}]\label{lem:vector-mean}
For given matrix 
\begin{align*}
\mG=[\vg_1;\dots;\vg_m]\in\BR^{m\times d},     
\end{align*}
we have~$\Norm{\mG - \vone\bg}^2 \leq \Norm{\mG}^2$, where $\bg=\frac{1}{m}\sum_{i=1}^m \vg_i$.
\end{lem}

\begin{lem}[{\citet[Lemma B.5]{li2024problem}}]
For given vectors 
\begin{align*}
\vv_1,\dots,\vv_m\in\BR^{1\times d} \quad\text{and}\quad \bv=\frac{1}{m}\sum_{i=1}^m\vv_i,     
\end{align*}
we have    
\begin{align}\label{eq:consensus-error-1}
    \Norm{\frac{1}{m}\sum_{i=1}^m \frac{\vv_i}{\Norm{\vv_i}} - \frac{\bv}{\Norm{\bv}}} 
\leq \frac{1}{m}\sum_{i=1}^m \frac{\Norm{\bv - \vv_i}}{\Norm{\bv}}. 
\end{align}
\end{lem}
\begin{proof}
This result is presented in the analysis \citet[Lemma B.5]{li2024problem}. 
Here, we provide the detailed proof for completeness. We have
\begin{align*}
& {\Norm{\frac{1}{m}\sum_{i=1}^m \frac{\vv_i}{\Norm{\vv_i}} - \frac{\bv}{\Norm{\bv}}} 
= \Norm{\frac{1}{m}\sum_{i=1}^m \frac{\vv_i}{||\vv_i||} - \frac{1}{m}\sum_{i=1}^m\frac{\vv_i}{||\bar \vv||}}} \\
{=} & {\Norm{\frac{1}{m}\sum_{i=1}^m \frac{(\Norm{\bv} - \Norm{\vv_i})\vv_i}{\Norm{\vv_i}\Norm{\bv}}}}  
\leq  \frac{1}{m}\sum_{i=1}^m\Norm{ \frac{(\Norm{\bv} - \Norm{\vv_i})\vv_i}{\Norm{\vv_i}\Norm{\bv}}}  \\
\leq & \frac{1}{m}\sum_{i=1}^m \frac{\left|\Norm{\bv} - \Norm{\bv_i}\right|\,\Norm{\vv_i}}{\Norm{\vv_i}\Norm{\bv}}  
= \frac{1}{m}\sum_{i=1}^m \frac{\left|\Norm{\bv} - \Norm{\vv_i}\right|}{\Norm{\bv}} \\
\leq & \frac{1}{m}\sum_{i=1}^m \frac{\Norm{\bv - \vv_i}}{\Norm{\bv}},
\end{align*}
where {the first equality is based on the fact $\bar\vv=\frac{1}{m}\sum_{i=1}^m \vv_i$}; the first inequality is based on vector AM-QM inequality; the second inequality is based on Cauchy--Schwarz inequality; the last step is based on the triangle inequality.
\end{proof}

\begin{lem}
The update of Algorithm \ref{alg:DNSGD}  implies $\bv^t = \bg^t$.   
\end{lem}
\begin{proof}
This result can be obtained by following \citet[Lemma 2]{luo2022complexity}. Here, we provide the detailed proof for completeness.
The update in line \ref{line:V0} and Proposition \ref{proposition_fastmix} imply $\bv^0=\bg^0$. Suppose it holds $\bv^t=\bg^t$, then the update in line \ref{line:Vt1} means
\begin{align*}
    \bv^{t+1} = \frac{1}{m}\sum_{i=1}^m (\vv_i^t + \vg_i^{t+1} - \vg_i^t) = \bv^t + \bg^{t+1} - \bg^t = \bg^{t+1}.
\end{align*}
Therefore, we finish the proof by induction.
\end{proof}

\begin{lem}[{\citet[Appendix A.1]{zhang2020improved}}]\label{lem:sym-smooth}
Under Assumptions \ref{asm:sym-1} and \ref{asm:heterogeneous}, we have
\begin{align}\label{eq:smooth-descent}
    f(\vy) \leq f(\vx) + \inner{\nabla f(\vx)}{\vy-\vx} + \frac{L_f+L_1\Norm{\nabla f(\vx)}}{2}\Norm{\vy-\vx}^2
\end{align}
for all $\vx,\vy\in\BR^d$ with $\Norm{\vy-\vx}\leq 1/L_1$.
\end{lem}
\begin{proof}
Proposition \ref{prop:global-l0l1} means the objective $f$ is $(L_f, L_1)$-smooth, 
then applying equation (29) of \citet[Appendix A.1]{zhang2020improved} directly achieves equation (\ref{eq:smooth-descent}).
\end{proof}

\section{The Proof of Proposition \ref{prop:global-l0l1}}\label{sec:proof-global-l0l1}
\begin{proof}    
For all $\vx,\vy\in\BR^d$ with $\Norm{\vx-\vy}\leq 1/L_1$, we have
\begin{align*}
   & \Norm{\nabla f(\vx) - \nabla f(\vy)}^2 \\
=  & \Norm{\frac{1}{m}\sum_{i=1}^m\big(\nabla f_i(\vx) - \nabla f_i(\vy)\big)}^2 \\
\leq & \frac{1}{m}\sum_{i=1}^m\Norm{\nabla f_i(\vx) - \nabla f_i(\vy)}^2 \\
\leq & \frac{1}{m}\sum_{i=1}^m (L_0 + L_1\Norm{\nabla f_i(\vx)})^2\Norm{\vx-\vy}^2 \\
\leq & (L_0 + L_1(\Norm{\nabla f_i(\vx)-\nabla f(\vx)} + \Norm{\nabla f(\vx)}))^2\Norm{\vx-\vy}^2 \\
= &  (L_0 +L_1\zeta + L_1\Norm{\nabla f(\vx)})^2\Norm{\vx-\vy}^2,
\end{align*}
where the first inequality is based on the Young's inequality; the second inequality is based on Assumption \ref{asm:sym-1}; the third inequality is based on Assumption \ref{asm:heterogeneous}.
Therefore, the function $f$ is $(L_f,L_1)$-smooth with $L_f=L_0+L_1\zeta$.
\end{proof}

\section{The Proof of Lemma \ref{lem:descent-mean-2}}\label{sec:desent-proof}

We first provide the upper bound for the difference of two mean vectors $\bx^{t+1}$ and $\bx^t$.
\begin{lem}\label{lem:diff-bar-x}
Under the setting of Theorem \ref{thm:main}, we have
\begin{align*}
\Norm{\bx^{t+1} - \bx^t}^2 \leq  \eta^2 \leq \frac{1}{4L_1^2}. 
\end{align*}
\end{lem}
\begin{proof}
The updates in lines \ref{line:update-V} and \ref{line:update-X} of Algorithm \ref{alg:DNSGD} imply
\begin{align}\label{eq:diff-bar-x}
    \bx^{t+1} =& \bx^t - \eta \bu^t 
    =  \bx^t - \eta \cdot \frac{1}{m}\sum_{i=1}^m \frac{\vv^t_i}{\Norm{\vv^t_i}},
\end{align}
where the first step is based on Proposition \ref{proposition_fastmix}. Therefore, we have
\begin{align}\label{eq:bx-eta}
\begin{split}    
   & \Norm{\bx^{t+1} - \bx^t}^2 
=  \eta^2\Norm{\frac{1}{m}\sum_{i=1}^m \frac{\vv^t_i}{\Norm{\vv^t_i}}}^2  \\
\leq & \frac{\eta^2}{m}\sum_{i=1}^m\Norm{ \frac{\vv^t_i}{\Norm{\vv^t_i}}}^2
=  \eta^2 \leq \frac{1}{4L_1^2},
\end{split}
\end{align}    
where the first inequality is based on the vector AM-QM inequality and the last step is based on the setting of $\eta$.
\end{proof}

We then connect the gradient $\nabla f(\bx^{t})$ to $\nabla f(\bx^{t+1})$ as follows.

\begin{lem}\label{lem:recursion-grad}
Under the setting of Theorem \ref{thm:main}, we have
\begin{align*}
    \Norm{\nabla f(\bx^{t+1})} \leq L_f\eta + \left(1+L_1\eta\right)\Norm{\nabla f(\bx^{t})}.
\end{align*}
\end{lem}
\begin{proof}    
We have
\begin{align}\label{eq:recursion-grad}
\begin{split}    
   & \Norm{\nabla f(\bx^{t+1})} \\
\leq & \Norm{\nabla f(\bx^{t+1})-\nabla f(\bx^{t})} +   \Norm{\nabla f(\bx^{t})} \\
\leq & \left(L_f + L_1\Norm{\nabla f(\bx^{t})}\right)\Norm{\bx^{t+1}-\bx^{t}} +   \Norm{\nabla f(\bx^{t})} \\
\leq & \left(L_f + L_1\Norm{\nabla f(\bx^{t})}\right)\eta +   \Norm{\nabla f(\bx^{t})} \\
= & L_f\eta + \left(1+L_1\eta\right)\Norm{\nabla f(\bx^{t})},
\end{split}
\end{align}
where the first step is based on triangle inequality; 
the second step is based on Proposition \ref{prop:global-l0l1};
and third last step is based on Lemma~\ref{lem:diff-bar-x}.
\end{proof}

Now we present the proof of Lemma \ref{lem:descent-mean-2}.

\begin{proof}
Applying Lemmas \ref{lem:sym-smooth}, \ref{lem:diff-bar-x} and \ref{lem:recursion-grad}, we have
\begin{align}  
\label{eq:descent-mean-2}\small\begin{split}
   &  f(\bx^{t+1}) - f(\bx^t)  \\
\overset{~(\ref{eq:smooth-descent})~}{\leq} &  \inner{\nabla f(\bx^t)}{\bx^{t+1}-\bx^t} + \frac{L_f+L_1\Norm{\nabla f(\bx^t)}}{2}\Norm{\bx^{t+1}-\bx^t}^2 \\
\overset{(\ref{eq:bx-eta})}{\leq} &  \inner{\nabla f(\bx^t)-\bv^t}{\bx^{t+1}-\bx^t}
+ \inner{\bv^t}{\bx^{t+1}-\bx^t} \\
& + \frac{\eta^2}{2}(L_f+L_1\Norm{\nabla f(\bx^t)}) \\
\overset{~(\ref{eq:diff-bar-x})~}{\leq} &  \Norm{\nabla f(\bx^t)-\bv^t}\Norm{\bx^{t+1}-\bx^t}
- \eta\inner{\bv^t}{\frac{1}{m}\sum_{i=1}^m \frac{\vv_i^t}{\Norm{\vv_i^t}}} \\
& + \frac{\eta^2}{2}(L_f+L_1\Norm{\nabla f(\bx^t)}) \\
\overset{(\ref{eq:bx-eta})}{\leq} &  \eta\Norm{\nabla f(\bx^t)-\bv^t}
+ \eta\inner{\bv^t}{\frac{\bv^t}{\Norm{\bv^t}} - \frac{\bv^t}{\Norm{\bv^t}} - \frac{1}{m}\sum_{i=1}^m \frac{\vv_i^t}{\Norm{\vv_i^t}}} \\
& + \frac{\eta^2}{2}(L_f+L_1\Norm{\nabla f(\bx^t)}) \\
~=~ &  \eta\Norm{\nabla f(\bx^t)-\bv^t} - \eta\Norm{\bv^t} 
+ \eta\inner{\bv^t}{\frac{\bv^t}{\Norm{\bv^t}} -  \frac{1}{m}\sum_{i=1}^m \frac{\vv_i^t}{\Norm{\vv_i^t}}} \\
& + \frac{\eta^2}{2}(L_f+L_1\Norm{\nabla f(\bx^t)}) \\
~\leq~ & \eta\Norm{\nabla f(\bx^t)-\bv^t} - \eta\Norm{\bv^t} 
+ \eta\Norm{\bv^t}\Norm{\frac{\bv^t}{\Norm{\bv^t}} -  \frac{1}{m}\sum_{i=1}^m \frac{\vv_i^t}{\Norm{\vv_i^t}}} \\
& + \frac{\eta^2L_f}{2} + \frac{\eta^2L_1}{2}\Norm{\nabla f(\bx^t)} \\
~\leq~ & \eta\Norm{\nabla f(\bx^t)-\bv^t} + \eta(\Norm{\nabla f(\bx^t)-\bv^t} - \Norm{\nabla f(\bx^t)}) \\
& + \eta\Norm{\bv^t}\Norm{\frac{\bv^t}{\Norm{\bv^t}} - \frac{1}{m}\sum_{i=1}^m \frac{\vv_i^t}{\Norm{\vv_i^t}}}
+ \frac{\eta^2L_f}{2} + \frac{\eta^2L_1}{2}\Norm{\nabla f(\bx^t)} \\
\overset{~(\ref{eq:consensus-error-1})~}{\leq} & 2\eta\Norm{\nabla f(\bx^t)-\bv^t} - \eta\left(1 - \frac{\eta L_1}{2}\right) \Norm{\nabla f(\bx^t)} \\
& + \frac{\eta}{m}\sum_{i=1}^m \Norm{\bv^t - \vv^t_i} + \frac{\eta^2L_f}{2} \\
~\leq~ & 2\eta\Norm{\nabla f(\bx^t)-\bv^t} - \eta\left(1 - \frac{\eta L_1}{2}\right)\Norm{\nabla f(\bx^t)} \\
& + \frac{\eta}{\sqrt{m}}\Norm{\mV^t-\vone\bv^t}
+ \frac{\eta^2L_f}{2}, 
\end{split}
\end{align}
where the first inequality is based on Lemma \ref{lem:sym-smooth} and the fact $\Norm{\bx^{t+1} - \bx^t} \leq 1/(2L_1)$ from Lemma~\ref{eq:diff-bar-x};
the forth to last line is based on Cauchy--Schwarz inequality; 
the third to last line is based on the triangle inequality;
the last line is based on the vector AM-QM inequality.

We bound the term $\Norm{\nabla f(\bx^t)-\bv^t}$ in equation (\ref{eq:descent-mean-2}) as
\begin{align}\label{eq:diff-grad-v0}
\small\begin{split}    
 & \Norm{\nabla f(\bx^t)-\bv^t}^2 
= \Norm{\frac{1}{m}\sum_{i=1}^m (\nabla f_i(\bx^t) - \vg_i^t))}^2 \\
\leq & 2\underbrace{\Norm{\frac{1}{m}\sum_{i=1}^m (\nabla f_i(\bx^t) - \nabla f_i(\vx_i^t))}^2}_{B_1} 
+ 2\underbrace{\Norm{\frac{1}{m}\sum_{i=1}^m (\nabla f_i(\vx_i^t)-\vg_i^t)}^2}_{B_2}. 
\end{split}
\end{align}

We bound term $B_1$ as
\begin{align}\label{eq:diff-grad-v1}
\begin{split}     
B_1= & \Norm{\frac{1}{m}\sum_{i=1}^m (\nabla f_i(\bx^t) - \nabla f_i(\vx_i^t))}^2 \\
\leq & \frac{1}{m}\sum_{i=1}^m\Norm{\nabla f_i(\bx^t) - \nabla f_i(\vx_i^t)}^2 \\
\leq &  \frac{1}{m}\sum_{i=1}^m 2\left(L_0^2 + L_1^2\Norm{\nabla f_i(\bx^t)}^2\right) \Norm{\bx^t-\vx_i^t}^2 \\
\leq &  \frac{1}{m}\sum_{i=1}^m 2\left(L_0^2 + L_1^2\left(\Norm{\nabla f(\bx^t)}+\zeta\right)^2\right) \Norm{\bx^t-\vx_i^t}^2 \\
\leq &  \frac{1}{m}\left(M_0 + M_1\Norm{\nabla f(\bx^t)}\right)^2\Norm{\mX^t-\vone\bx^t}^2,
\end{split}
\end{align}
where we denote $M_0=\sqrt{2(L_0^2 +L_1^2\zeta^2)}$ and $M_1=\sqrt{2}L_1$. 
Here, the first inequality is based on vector AM-QM inequality; 
the second inequality is based on Young's inequality, Assumption, \ref{asm:sym-1} {and the fact $\Norm{\bx^t-\vx_i^t}\leq1/L_1$ from equation (\ref{eq:bound-xti-barx})};
the third inequality is based on Assumption \ref{asm:heterogeneous}. 

For given $\{\vx_i^t\}_{i=1}^m$, we bound the expectation of term $B_2$ as
\begin{align}\label{eq:diff-grad-v2}
\begin{split}    
\BE[B_2] = & \BE\Norm{\frac{1}{m}\sum_{i=1}^m (\vg_i^t - \nabla f_i(\vx_i^t))}^2  \\
= & \frac{1}{m^2}\sum_{i=1}^m {\mathbb E}\left\|\vg_i^t - \nabla f_i(\vx_i^t)\right\|^2  \\
& + \frac{1}{m^2}\sum_{i=1}^m\sum_{j\neq i} {\mathbb E}\left[\left\langle \vg_i^t - \nabla f_i(\vx_i^t), \vg_j^t - \nabla f_j(\vx_j^t)\right\rangle\right] \\
= & \frac{1}{m^2}\sum_{i=1}^m \BE\Norm{\vg_i^t - \nabla f_i(\vx_i^t)}^2  \\
\leq & \frac{1}{m^2}\sum_{i=1}^m \frac{\sigma^2}{b}  = \frac{\sigma^2}{mb},
\end{split}
\end{align}
where the second step is based on the fact that $\vg_i^t$ and $\vg_j^t$ are independent for $i\neq j$ which implies ${\mathbb E}[\langle \vg_i^t - \nabla f_i(\vx_i^t), \vg_j^t - \nabla f_i(\vx_j^t)\rangle]=0$;
the inequality is based on the setting $\vg_i^t=\frac{1}{b}\sum_{k=1}^b \nabla F_i(\vx_i^{{t}};\vxi_{i,k}^{{t}})$ in line \ref{line:Gt} of Algorithm \ref{alg:DNSGD} and Assumption \ref{asm:sfo}.

Combining results (\ref{eq:diff-grad-v0}), (\ref{eq:diff-grad-v1}) and (\ref{eq:diff-grad-v2}), we have
\begin{align}\label{eq:diff-grad-v}
\begin{split}    
  & \BE\Norm{\nabla f(\bx^t)-\bv^t} \\
\leq & \sqrt{\frac{2}{m}\left(M_0 + M_1\Norm{\nabla f(\bx^t)}\right)^2\Norm{\mX^t-\vone\bx^t}^2 + \frac{\sigma^2}{mb}} \\
\leq & \sqrt{\frac{2}{m}}\left(M_0 + M_1\Norm{\nabla f(\bx^t)}\right)\Norm{\mX^t-\vone\bx^t} + \frac{\sigma}{\sqrt{mb}},
\end{split}
\end{align}
for given $\mX^t$, where the first step is based on the Jensen's inequality 
and the second step is based on the elementary inequality $\sqrt{a^2+b^2}\leq a+b$ for $a,b>0$.

Combining equations (\ref{eq:descent-mean-2}) and (\ref{eq:diff-grad-v}), we finish the proof.
\end{proof}

\section{The Proof of Lemma \ref{lem:recursion-X}}
\begin{proof}    
Recall that line \ref{line:U} of Algorithm \ref{alg:DNSGD} set
\begin{align*}
\mU^t = \begin{bmatrix}
        \dfrac{\vv^t_1}{\Norm{\vv^t_1}} & \dots & \dfrac{\vv^t_1}{\Norm{\vv^t_m}}
\end{bmatrix}^\top,    
\end{align*}
which implies $\Norm{\mU^t}^2=m$.
Based on equation (\ref{eq:diff-bar-x}), we have
\begin{align}\label{eq:recursion-x}
\begin{split}    
  &  \Norm{\mX^{t+1}-\vone\bx^{t+1}} \\
= &  \Norm{\AG(\mX^t - \eta \mU^t, K) - \frac{1}{m}\vone\vone^\top\AG(\mX^t - \eta \mU^t, K)} \\
\leq & \rho\Norm{\mX^t - \eta \mU^t- \frac{1}{m}\vone\vone^\top(\mX^t - \eta \mU^t)} \\
= & \rho\Norm{\mX^t - \eta \mU^t- \vone(\bx^t - \eta \bu^t)} \\
\leq & \rho\left(\Norm{\mX^t - \vone\bx^t} + \eta\Norm{\mU^t - \vone\bu^t}\right) \\
\leq & \rho\left(\Norm{\mX^t - \vone\bx^t} + m\eta\right),
\end{split}
\end{align}
where the first inequality is based on Proposition \ref{proposition_fastmix}; 
the second inequality is based on the triangle inequality;
the last step applies Lemma \ref{lem:vector-mean} which leads to
\begin{align*}
    \Norm{\mU^t - \vone\bu^t}^2 
\leq \Norm{\mU^t}^2 = m.
\end{align*}
Based on equation (\ref{eq:recursion-x}), we have
\begin{align*}
    \Norm{\mX^t-\vone\bx^t} 
\leq & \rho^t\left(\Norm{\mX^0 - \vone\bx^0} - \frac{\rho m\eta}{1-\rho}\right) + \frac{\rho m\eta}{1-\rho} \leq \frac{\rho m\eta}{1-\rho}, 
\end{align*}
where the second inequality is based on the initialization such that $\vx_i^0=\bar\vx^0$ for each $i$.

\end{proof}

\section{The Proof of Lemma \ref{lem:recursion-V}}
\begin{proof}    
We denote
\begin{align*}
    \nabla \vf(\mX^{t})
= \begin{bmatrix}
    \nabla f_1(\vx_1^{t})^\top \\ \vdots \\
    \nabla f_m(\vx_m^{t})^\top \\
\end{bmatrix},~
\nabla \vf(\mX^{t+1})
= \begin{bmatrix}
    \nabla f_1(\vx_1^{t+1})^\top \\ \vdots \\
    \nabla f_m(\vx_m^{t+1})^\top \\
\end{bmatrix}\in\BR^{m\times d}.
\end{align*}
The update in line \ref{line:Vt1} of Algorithm \ref{alg:DNSGD} implies
\begin{align}\label{eq:recursion-v}
\begin{split} 
   & \Norm{\mV^{t+1} - \vone\bv^{t+1}} \\
= & \Bigg\|\AG(\mV^t + \mG^{t+1} - \mG^t),K) \\
&~~~~- \frac{1}{m}\vone\vone^\top\AG(\mV^t + \mG^{t+1} - \mG^t),K)\Bigg\| \\
\leq & \rho\Norm{\mV^t + \mG^{t+1} - \mG^t - \frac{1}{m}\vone\vone^\top(\mV^t + \mG^{t+1} -\mG^{t})} \\
\leq & \rho\Norm{\mV^t - \vone\bv^t} + \rho\Norm{\mG^{t+1} - \mG^t - \frac{1}{m}\vone\vone^\top(\mG^{t+1} - \mG^{t})} \\
\leq & \rho\Norm{\mV^t - \vone\bv^t} + \rho\Norm{\mG^{t+1} - \mG^{t}} \\
\leq & \rho\Norm{\mV^t - \vone\bv^t}  + \rho\underbrace{\Norm{\nabla \vf(\mX^{t+1}) - \nabla \vf(\mX^t)}}_{C_1} \\ 
& + \rho\underbrace{\Norm{\mG^{t+1} - \nabla \vf(\mX^{t+1})}}_{C_2} + \rho\underbrace{\Norm{\mG^{t} - \nabla \vf(\mX^t)}}_{C_3}, 
\end{split}
\end{align}
where the first three step follows the derivation of equation (\ref{eq:recursion-x});
the fourth step is based on Lemma~\ref{lem:vector-mean};
the last step is based on triangle inequality.

We first consider term $C_1$.
The setting of $K$ in this lemma implies 
\begin{align*}
\rho \leq \frac{1}{1+m \eta L_1},
\end{align*}
which leads to
\begin{align}\label{eq:bound-xti-barx}
\Norm{\vx_t^i-\bx^t} \leq \Norm{\mX^t-\vone\bx^t} \leq\frac{\rho m\eta}{1-\rho} \leq \frac{1}{L_1}. 
\end{align}
Therefore, we have
\begin{align*}
& \Norm{\nabla f(\vx^{t+1}_i) - \nabla f(\vx^t_i)} \\
\leq &  \Norm{\nabla f(\vx^{t+1}_i) - \nabla f(\bx^{t+1})} \\
& + \Norm{\nabla f(\bx^{t+1}) - \nabla f(\bx^t)} + \Norm{\nabla f(\bx^t_i) - \nabla f(\vx^t_i)} \\
\leq & \left(L_f+L_1\Norm{\nabla f(\bx^{t+1})}\right)\Norm{\vx^{t+1}_i - \bx^{t+1}} \\
& + \left(L_f+L_1\Norm{\nabla f(\bx^t)}\right)\Norm{\bx^{t+1} - \bx^t}  
  + \left(L_f+L_1\Norm{\nabla f(\bx^t)}\right)\Norm{\vx^{t}_i - \bx^{t}},
\end{align*}
where the first step is based on the triangle inequality and the last step is based on Proposition \ref{prop:global-l0l1} and equation (\ref{eq:bound-xti-barx}).
Summing over above result with $i=1,\dots,m$, we have
\begin{align}\label{eq:diff-grad}
 \begin{split}    
 C_1=&\sum_{i=1}^m\Norm{\nabla f(\vx^{t+1}_i) - \nabla f(\vx^t_i)}^2 \\
\leq & 3\left(L_f+L_1\Norm{\nabla f(\bx^{t+1})}\right)^2\sum_{i=1}^m\Norm{\vx^{t+1}_i - \bx^{t+1}}^2 \\
 & + 3\left(L_f+L_1\Norm{\nabla f(\bx^t)}\right)^2\sum_{i=1}^m\Norm{\bx^{t+1} - \bx^t}^2 \\
 & + 3\left(L_f+L_1\Norm{\nabla f(\bx^t)}\right)^2\sum_{i=1}^m\Norm{\vx^{t}_i - \bx^{t}}^2 \\
\leq & 3  \left(L_f+L_1\Norm{\nabla f(\bx^{t+1})}\right)^2\Norm{\mX^{t+1} - \vone\bx^{t+1}}^2 \\
 & + 3m\left(L_f+L_1\Norm{\nabla f(\bx^t)}\right)^2\eta^2 \\ 
 & + 3\left(L_f+L_1\Norm{\nabla f(\bx^t)}\right)^2\Norm{\mX^{t} - \vone\bx^{t}}^2,
\end{split}
\end{align}
where the last step is based on Lemma \ref{lem:diff-bar-x}.

Substituting equations (\ref{eq:recursion-x}) and (\ref{eq:recursion-grad}) into equation (\ref{eq:diff-grad}), we have
\begin{align}\label{eq:diff-grad-2}\small
\begin{split}    
& \Norm{\nabla \vf(\mX^{t+1}) - \nabla \vf(\mX^t)}^2 \\
\leq & 3 \rho^2 \left(L_f+L_1\left(L_f\eta + (L_1\eta+1)\Norm{\nabla f(\bx^{t})}\right)\right)^2\left(\Norm{\mX^t - \vone\bx^t} + m\eta\right)^2 \\
 &  + 3\left(L_f+L_1\Norm{\nabla f(\bx^t)}\right)^2\Norm{\mX^{t} - \vone\bx^{t}}^2 \\
 & + 3m\left(L_f+L_1\Norm{\nabla f(\bx^t)}\right)^2\eta^2 \\
\leq & \left(M_2 + M_3\Norm{\nabla f(\bx^{t})}\right)^2 \left(\Norm{\mX^t - \vone\bx^t} + m\eta\right)^2,
\end{split}
\end{align}
where $M_2 = {3}(\rho+1)L_f+{3}\rho L_fL_1\eta$ and 
$M_3 = {3}\rho L_1(L_1\eta+1)+{3}L_1$.

For the term $C_2$, we have
\begin{align}\label{eq:sub-G}
\begin{split}    
  & \BE\left[\Norm{\mG^{t+1} - \nabla \vf(\mX^{t+1})}^2\right]  \\
= & \BE\left[\sum_{i=1}^m \Norm{\vg_i^{t+1} - \nabla f_i(\vx_i^{t+1})}^2\right] \\
= & \BE\left[\sum_{i=1}^m \Norm{
 \frac{1}{b}\sum_{k=1}^b \nabla F_i(\vx_i^{t+1};\vxi_{i,k}^{t+1}) - \nabla f_i(\vx_i^{t+1})}^2\right]    \\
\leq & \frac{m\sigma^2}{b},
\end{split}
\end{align}
for given $\{\vx_i^{t+1}\}_{i=1}^m$, 
where the inequality is based on Assumption \ref{asm:sfo}.

Consequently, we have
\begin{align}\label{eq:gtp1}
\begin{split}    
 C_2 = & \BE\left[\Norm{\mG^{t+1} - \nabla \vf(\mX^{t+1})}\right]  
=  \sqrt{(\BE\Norm{\mG^{t+1} - \nabla \vf(\mX^{t+1})})^2} \\
\leq & \sqrt{\BE\left[\Norm{\mG^{t+1} - \nabla \vf(\mX^{t+1})}^2\right]}
\overset{(\ref{eq:sub-G})}{\leq}  \frac{\sqrt{m}\sigma}{\sqrt{b}},
\end{split}
\end{align}
where the first inequality is based on Jensen's inequality.

Similarly, the term $C_3$ holds
\begin{align}\label{eq:gt}
  C_3 =  \BE\left[\Norm{\mG^{t} - \nabla \vf(\mX^{t})}\right] 
\leq \frac{\sqrt{m}\sigma}{\sqrt{b}}.
\end{align}

Combining the results of (\ref{eq:recursion-v}), (\ref{eq:diff-grad-2}), (\ref{eq:gtp1}), and (\ref{eq:gt}), we have
\begin{align*}\small
\begin{split} 
   & \BE\left[\Norm{\mV^{t+1} - \vone\bv^{t+1}}\right] \\
\leq & \rho\Norm{\mV^t - \vone\bv^t} + \rho\left(M_2 + M_3\Norm{\nabla f(\bx^{t})}\right) \left(\Norm{\mX^t - \vone\bx^t} + m\eta\right) + \frac{2\rho\sqrt{m}\sigma}{\sqrt{b}} \\
\leq & \rho\Norm{\mV^t - \vone\bv^t} + \rho\left(M_2 + M_3\Norm{\nabla f(\bx^{t})}\right)\Norm{\mX^t - \vone\bx^t} \\
& + \rho\left(M_2 + M_3\Norm{\nabla f(\bx^{t})}\right) m\eta + \frac{2\rho\sqrt{m}\sigma}{\sqrt{b}},
\end{split}
\end{align*}
for given $\{\vx_i^{t+1}\}_{i=1}^m$ and $\{\vv_i^{t+1}\}_{i=1}^m$, which finishes the proof.
\end{proof}

\section{The proof of Theorem \ref{thm:main}}

For given $\mX^t$ and $\mV^t$, 
we apply Lemmas \ref{lem:descent-mean-2}, \ref{lem:recursion-X}, \ref{lem:recursion-V}, and \ref{lem:recursion-grad} to achieve
\begin{align*}    
& \BE[\Phi^{t+1}] \\
= & \BE\left[f(\bx^{t+1}) + \frac{3\eta}{\sqrt{m}}\left(M_0 + M_1\Norm{\nabla f(\bx^{t+1})}\right)\Norm{\mX^{t+1}-\vone\bx^{t+1}}\right] \\
& \quad+ \frac{2\eta}{\sqrt{m}}\BE\left[\Norm{\mV^{t+1}-\vone\bv^{t+1}}\right] \\
\leq & f(\bx^t) - \eta\left(1 - \frac{\eta L_1}{2}\right)\Norm{\nabla f(\bx^t)} \\
&  + \frac{2\sqrt{2}\eta}{\sqrt{m}}\left(M_0 + M_1\Norm{\nabla f(\bx^t)}\right)\Norm{\mX^t-\vone\bx^t} \\
& + \frac{\eta}{\sqrt{m}}\Norm{\mV^t - \vone\bv^t} + \frac{2\eta\sigma}{\sqrt{mb}}
+ \frac{\eta^2L_f}{2} \\
&  + \frac{3\eta\rho}{\sqrt{m}}\left(M_0 + M_1\left(L_f\eta + \left(1+L_1\eta\right)\Norm{\nabla f(\bx^{t})}\right)\right)\left(\Norm{\mX^t - \vone\bx^t} + m\eta\right) \\
 &  + \frac{2\eta\rho}{\sqrt{m}}\Bigg(\Norm{\mV^t - \vone\bv^t} + \left(M_2 + M_3\Norm{\nabla f(\bx^{t})}\right)\Norm{\mX^t - \vone\bx^t} \\
 & \qquad\quad\quad + \left(M_2 + M_3\Norm{\nabla f(\bx^{t})}\right) m\eta + \frac{2\sqrt{m}\sigma}{\sqrt{b}}\Bigg) \\
 = & f(\bx^t)  + \left(\frac{2\sqrt{2}M_0\eta}{\sqrt{m}} + \rho\Big(\frac{3\eta(M_0 + M_1L_f\eta)}{\sqrt{m}} + \frac{2\eta M_2}{\sqrt{m}}\Big)\right)\Norm{\mX^t-\vone\bx^t} \\
 & + \left(\frac{2\sqrt{2}M_1\eta}{\sqrt{m}} + \rho\left(\frac{3\eta M_1(1 + L_1\eta)}{\sqrt{m}} + \frac{2\eta M_3}{\sqrt{m}}\right)\right)\Norm{\nabla f(\bx^t)}\Norm{\mX^t-\vone\bx^t} \\
 & + \left(\frac{\eta}{\sqrt{m}}+\frac{2\rho\eta}{\sqrt{m}}\right)\Norm{\mV^t-\vone\bv^t} \\
 & - \eta\left(1 - \frac{\eta L_1}{2} - \rho\Big(3\sqrt{m}\eta M_1(1 + L_1\eta) + 2\sqrt{m}\eta M_3\Big)\right)\Norm{\nabla f(\bx^t)} \\
 & + \frac{2\eta\sigma}{\sqrt{mb}}  + \frac{\eta^2L_f}{2}  + \rho\Big(3\sqrt{m}\eta^2(M_0 + M_1L_f\eta) + 2\sqrt{m}\eta^2M_2{+\frac{4\eta\sigma}{\sqrt{b}}}\Big)  \\
\leq & f(\bx^t)  + \frac{3\eta}{\sqrt{m}}\left(M_0 + M_1\Norm{\nabla f(\bx^t)}\right)\Norm{\mX^t-\vone\bx^t} + \frac{2\eta}{\sqrt{m}}\Norm{\mV^t-\vone\bv^t}   \\
& - \eta\left(\frac{7}{8} - \frac{\eta L_1}{2}\right)\Norm{\nabla f(\bx^t)} + \frac{3\eta^2L_f}{4} \\
= & \Phi^t - \eta\left(\frac{7}{8} - \frac{\eta L_1}{2}\right)\Norm{\nabla f(\bx^t)} + \frac{3\eta^2L_f}{4} \\
\leq & \Phi^t - \frac{5\eta}{8}\Norm{\nabla f(\bx^t)} + \frac{3\eta^2L_f}{4},
\end{align*}
where the second inequality requires the condition
\begin{align*}
\begin{cases}    
 2\sqrt{2}M_0 + \rho(3(M_0 + M_1L_f\eta) + 2 M_2) \leq 3M_0,\\[0.1cm]
 2\sqrt{2}M_1 + \rho\left(3 M_1(1 + L_1\eta) + 2 M_3\right) \leq 3M_1, \\[0.1cm]
 \eta + 2\rho\eta \leq 2\eta,  \\[0.15cm]
 \rho\Big(3\sqrt{m}\eta M_1(1 + L_1\eta) + 2\sqrt{m}\eta M_3\Big) \leq \dfrac{1}{8}, \\[0.25cm]
 \dfrac{2\sigma}{\sqrt{mb}}  + \rho\Big(3\sqrt{m}\eta(M_0 + M_1L_f\eta) + 2\sqrt{m}\eta M_2 {+\dfrac{4\eta\sigma}{\sqrt{b}}}\Big) \leq \dfrac{\eta L_f}{4}
\end{cases}
\end{align*}
that holds by taking $b\geq\left\lceil {256\sigma^2}/(m\eta^2L_f^2)\right\rceil$
and 
\begin{align*}    
\rho \leq \min\Bigg\{&\frac{1}{1+m \eta L_1},
  \frac{(3-2\sqrt{2})M_0}{3(M_0 + M_1L_f\eta) + 2 M_2}, \frac{(3-2\sqrt{2})M_1}{3 M_1(1 + L_1\eta) + 2 M_3},\\
  &  \frac{1}{2}, \dfrac{1}{8(3\sqrt{m}\eta M_1(1 + L_1\eta) + 2\sqrt{m}\eta M_3)}, \\
 & \dfrac{\eta L_f}{8(3\sqrt{m}\eta(M_0 + M_1L_f\eta) + 2\sqrt{m}\eta M_2{+{4\eta\sigma}/{\sqrt{b}}\,})} \Bigg\}.
\end{align*}
Recall that we have set
$L_1\eta \leq 1/2$, $M_0=\Theta(L_f)$, $M_1=\Theta(L_1)$, $M_2\leq \Theta(L_f)$, and $M_3\leq \Theta(L_1)$,
which means we require 
\begin{align*}    
\rho \leq \fO\left(\frac{1}{{\sqrt{m}(1+\eta)}}\right).
\end{align*}
This can be achieved by taking  
\begin{align*}
K=\fO\left(\frac{\log{{(\sqrt{m}(1+\eta))}}}{\sqrt{\gamma}}\right).    
\end{align*}

Following the setting $\eta = \min\left\{\epsilon/(4L_f+1), 1/(2L_1)\right\}$, we have
\begin{align*}    
\BE\left[\Norm{\nabla f(\bx^t)}\right]
\leq \frac{\BE[8(\Phi^t - \Phi^{t+1})]}{5\eta}  + \frac{6\eta L_f}{5}.
\end{align*}

Taking the average on above inequalities with $t=0,\dots,T-1$, we have
\begin{align*}
& \BE\left[\frac{1}{T}\sum_{t=0}^{T-1}\Norm{\nabla f(\bx^t)}\right] 
\leq \BE\left[\frac{8(\Phi^0 - \Phi^{T})}{5\eta T}  + \frac{6\eta L_f }{5}\right]
\leq  \frac{8\Delta_{\Phi}}{5\eta T}  + \frac{6\eta L_f}{5},
\end{align*}
where the last step holds due to
\begin{align*}\small
\begin{split}    
 & \BE[\Phi^0 - \Phi^T] \\
= & \BE\Bigg[f(\bx^0) + \frac{3\eta}{\sqrt{m}}\left(M_0 + M_1\Norm{\nabla f(\bx^0)}\right)\Norm{\mX^0-\vone\bx^0} + \frac{2\eta}{\sqrt{m}}\Norm{\mV^0-\vone\bv^0} \\
& - \left(f(\bx^T) + \frac{3\eta}{\sqrt{m}}\left(M_0 + M_1\Norm{\nabla f(\bx^T)}\right)\Norm{\mX^T-\vone\bx^T} + \frac{2\eta}{\sqrt{m}}\Norm{\mV^T-\vone\bv^T}\right)\Bigg] \\
\leq & f(\bx^0) - \inf_{\vx\in\BR^d} f(\vx)  + \frac{2\eta}{\sqrt{m}}\BE\left[\Norm{\mV^0-\vone\bv^0}\right] = \Delta_{\Phi}.
\end{split}
\end{align*}
We consider the following case for the choice of $\eta$:
\begin{enumerate}
    \item In the case of $\epsilon/(4L_f+1)\leq 1/(2L_1)$, we follow the setting 
    \begin{align}\label{eq:setting-1}
    \begin{split}        
        & \eta=\frac{\epsilon}{4L_f+1}, \quad
        b \geq \left\lceil \frac{256(4L_f+1)^2\sigma^2}{mL_f^2\epsilon^2}\right\rceil, \\
         & \text{and} \quad
        T \geq \frac{8\Delta_{\Phi}}{\eta\epsilon} = \frac{8(4L_f+1)\Delta_{\Phi}}{\epsilon^2}
    \end{split}
    \end{align}
    to obtain 
    \begin{align*}
    & \BE\left[\frac{1}{T}\sum_{t=0}^{T-1}\Norm{\nabla f(\bx^t)}\right]  
    \leq  \frac{\epsilon}{5}  + \frac{3\epsilon}{10} = \frac{\epsilon}{2}.
    \end{align*}
    \item In the case of $1/(2L_1)\leq\epsilon/(4L_f+1)$, we follow the setting
    \begin{align}\label{eq:setting-2}
        \eta=\frac{1}{2L_1},\quad
        b \geq \left\lceil \frac{1024L_1^2\sigma^2}{m L_f^2}\right\rceil,\quad\text{and}\quad
        T \geq \frac{16L_1\Delta_{\Phi}}{\epsilon}.
    \end{align}
    to obtain
    \begin{align*}
    & \BE\left[\frac{1}{T}\sum_{t=0}^{T-1}\Norm{\nabla f(\bx^t)}\right]  
    \leq \frac{\epsilon}{5} + \frac{3\epsilon}{10} = \frac{\epsilon}{2}.
    \end{align*}
\end{enumerate}
Hence, we finish the proof.

\section{The Proof of Corollary \ref{cor:main}}
\begin{proof}
We first provide the upper bound for $\BE\left[\Norm{\mV^0-\vone\bv^0}\right]$.
We let $\bg^0=\frac{1}{m}\vone^\top\mG^0$, then we have
\begin{align*}
   & \BE\left[\Norm{\mV^0-\vone\bv^0}^2\right] 
\leq  \BE\left[\hat\rho^2\Norm{\mG^0-\vone\bg^0}^2\right]  \\
\leq & \BE\left[\hat\rho^2\Norm{\mG^0}^2\right]  
=   \hat\rho^2\sum_{i=1}^m \BE\Norm{\frac{1}{b}\sum_{k=1}^b \nabla F_i(\vx_i^0;\vxi_{i,k}^0)}^2 \\
=  &  2\hat\rho^2\sum_{i=1}^m \left(\BE\left[\Norm{\frac{1}{b}\sum_{k=1}^b \nabla F_i(\vx_i^0;\vxi_{i,k}^0) - \nabla f_i(\vx_i^0)}^2 + \Norm{\nabla f_i(\vx_i^0)}^2\right]\right) \\
=  &  2\hat\rho^2\sum_{i=1}^m \left(\frac{\sigma^2}{b} + \Norm{\nabla f_i(\vx_i^0)}^2\right) \\
=  &  2\hat\rho^2 \left(\frac{m\sigma^2}{b} + \sum_{i=1}^m\Norm{\nabla f_i(\vx_i^0)}^2\right). 
\end{align*}
Hence, we require
\begin{align*}
    \hat\rho    \leq  \frac{\sqrt{m}\Delta_f}{2\sqrt{2}\eta\sqrt{{m\sigma^2}/{b} + \sum_{i=1}^m\Norm{\nabla f_i(\vx_i^0)}^2}}
\end{align*}
to guarantee
\begin{align}\label{eq:Delta_f}
\BE\left[\Norm{\mV^0-\vone\bv^0}\right]
\leq \sqrt{\BE\left[\Norm{\mV^0-\vone\bv^0}^2\right]}
\leq \frac{\sqrt{m}\Delta_f}{2\eta},
\end{align}
where the first step is based on Jensen's inequality.
Consequently, we have
\begin{align*}
\Delta_{\Phi} = f(\bx^0) - \inf_{\vx\in\BR^d} f(\vx) + \frac{2\eta}{\sqrt{m}}\BE\left[\Norm{\mV^0-\vone\bv^0}\right] \leq  2\Delta_f
\end{align*}
and the requirement of above $\hat\rho$ can be satasfied by taking
\begin{align*}
\hat K \geq \left\lceil\frac{1}{\sqrt{\gamma}}\log\frac{2\sqrt{2}\eta\sqrt{{m\sigma^2}/{b} + \sum_{i=1}^m\Norm{\nabla f_i(\vx_i^0)}^2}}{\sqrt{m}\Delta_f}\right\rceil
= \tilde\fO\left(\frac{1}{\sqrt{\gamma}}\right).
\end{align*}
Therefore, the term $\Delta_{\Phi}$ in the setting of $T$ presented in Theorem \ref{thm:main} can be replaced by $2\Delta_f$.

We then consider how to making the gradient small on each node.
Specifically, we have
\begin{align*}
& \BE\left[\Norm{\nabla f(\hat\vx_i)}\right] \\
= & \BE\left[\frac{1}{T}\sum_{t=0}^{T-1}\Norm{\nabla f(\vx_i^t)}\right] \\
\leq & \BE\left[\frac{1}{T}\sum_{t=0}^{T-1}\Norm{\nabla f(\vx_i^t)-\nabla f(\bx^t)} + \frac{1}{T}\sum_{t=0}^{T-1}\Norm{\nabla f(\bx^t)}\right] \\
\leq & \BE\left[\frac{1}{T}\sum_{t=0}^{T-1}\big(L_f+L_1\Norm{\nabla f(\bx^t)}\big)\Norm{\vx_i^t - \bx^t} + \frac{1}{T}\sum_{t=0}^{T-1}\Norm{\nabla f(\bx^t)}\right] \\
\leq & \BE\left[\frac{1}{T}\sum_{t=0}^{T-1}\big(L_f+L_1\Norm{\nabla f(\bx^t)}\big)\Norm{\mX^t - \vone\bx^t
}\right] + \frac{\epsilon}{2} \\
\leq & \BE\left[\frac{1}{T}\sum_{t=0}^{T-1}\frac{\rho m\eta\big(L_f+L_1\Norm{\nabla f(\bx^t)}\big)}{1-\rho}\right] + \frac{\epsilon}{2} \\
= & \frac{\rho m\eta\left(L_f+\BE\left[\frac{L_1}{T}\sum_{t=0}^{T-1}\Norm{\nabla f(\bx^t)}\right]\right)}{1-\rho}+ \frac{\epsilon}{2} \\
\leq & \frac{\rho m\eta\big(L_f+L_1\epsilon/2
\big)}{1-\rho}+ \frac{\epsilon}{2} 
\leq \epsilon,
\end{align*}
where the first step is based on the triangle inequality;
the second step is based on Proposition \ref{prop:global-l0l1} and equation (\ref{eq:bound-xti-barx});
the third step is based on Theorem \ref{thm:main};
the fourth step is based on Lemma \ref{lem:recursion-X};
the last line requires
\begin{align*}
    \rho\leq\frac{\epsilon}{2m\eta L_f + (m\eta L_1+1)\epsilon},
\end{align*}
which can be satisfies by taking $K=\tilde\fO(\sqrt{1/\gamma})$. 

In summary, the communication complexity is
\begin{align*}
   \hat K + TK = \tilde\fO\left(\left(\frac{L_f}{\epsilon^2} + \frac{L_1}{\epsilon}\right)\frac{\Delta_f}{\sqrt{\gamma}}\right).
\end{align*}
Following equations (\ref{eq:setting-1}), (\ref{eq:setting-2}), and (\ref{eq:Delta_f}), the sample complexity on each node is
\begin{align*}
Tb = \fO\left(\frac{1}{m}\left(\frac{L_f\sigma^2\Delta}{\epsilon^4} + \frac{\sigma^2}{\epsilon^2} + \frac{L_1^3\sigma^2\Delta}{L_f^2 \epsilon} + \frac{L_1^2\sigma^2 }{L_f^2} \right)\right), 
\end{align*}    
where the terms $\sigma^2/\epsilon^2$ and $L_1^2\sigma^2/L_f^2$ come from the initialization on $\mG^0$ in line \ref{line:G0} of Algorithm~\ref{alg:DNSGD}.
\end{proof}

\end{document}

%% file: math.tex
\usepackage{amsmath}\allowdisplaybreaks
\usepackage{amsfonts,bm}

\def\eqref#1{equation~(\ref{#1})}

\def\1{\bf{1}}

\newcommand{\Norm}[1]{\left\| #1 \right\|}

\def\inner#1#2{\left\langle #1, #2 \right\rangle}

\def\vone{{\bf{1}}}

\def\vf{{\bf{f}}}
\def\vg{{\bf{g}}}

\def\vu{{\bf{u}}}
\def\vv{{\bf{v}}}

\def\vx{{\bf{x}}}
\def\vy{{\bf{y}}}

\def\fD{{\mathcal{D}}}

\def\fO{{\mathcal{O}}}

\def\BE{{\mathbb{E}}}

\def\BN{{\mathbb{N}}}

\def\BR{{\mathbb{R}}}

\def\mG {{\bf G}}

\def\mI {{\bf I}}

\def\mU {{\bf U}}
\def\mV {{\bf V}}
\def\mW {{\bf W}}
\def\mX {{\bf X}}
\def\mY {{\bf Y}}

\usepackage{amsthm}
\theoremstyle{plain}
\newtheorem{thm}{Theorem}

\newtheorem{lem}{Lemma}
\newtheorem{asm}{Assumption}
\newtheorem{remark}{Remark}
\newtheorem{cor}{Corollary}
\newtheorem{prop}{Proposition}

\makeatletter
\def\Ddots{\mathinner{\mkern1mu\raise\p@
\vbox{\kern7\p@\hbox{.}}\mkern2mu
\raise4\p@\hbox{.}\mkern2mu\raise7\p@\hbox{.}\mkern1mu}}
\makeatother

\makeatletter
\newcommand*{\rom}[1]{\expandafter\@slowromancap\romannumeral #1@}
\makeatother

\def\bx{{\bar\vx}}
\def\bg{{\bar\vg}}

\def\bu{{\bar\vu}}
\def\bv{{\bar\vv}}
\def\vxi{{\bm\xi}}

\def\AG{{\rm AccGossip}}

%% file: sample-sigconf.bbl

\begin{thebibliography}{59}


\ifx \showCODEN    \undefined \def \showCODEN     #1{\unskip}     \fi
\ifx \showISBNx    \undefined \def \showISBNx     #1{\unskip}     \fi
\ifx \showISBNxiii \undefined \def \showISBNxiii  #1{\unskip}     \fi
\ifx \showISSN     \undefined \def \showISSN      #1{\unskip}     \fi
\ifx \showLCCN     \undefined \def \showLCCN      #1{\unskip}     \fi
\ifx \shownote     \undefined \def \shownote      #1{#1}          \fi
\ifx \showarticletitle \undefined \def \showarticletitle #1{#1}   \fi
\ifx \showURL      \undefined \def \showURL       {\relax}        \fi
\providecommand\bibfield[2]{#2}
\providecommand\bibinfo[2]{#2}
\providecommand\natexlab[1]{#1}
\providecommand\showeprint[2][]{arXiv:#2}

\bibitem[Arioli and Scott(2014)]%
        {arioli2014chebyshev}
\bibfield{author}{\bibinfo{person}{Mario Arioli} {and} \bibinfo{person}{Jennifer Scott}.} \bibinfo{year}{2014}\natexlab{}.
\newblock \showarticletitle{Chebyshev acceleration of iterative refinement}.
\newblock \bibinfo{journal}{\emph{Numerical Algorithms}} \bibinfo{volume}{66}, \bibinfo{number}{3} (\bibinfo{year}{2014}), \bibinfo{pages}{591--608}.
\newblock


\bibitem[Assran et~al\mbox{.}(2019)]%
        {assran2019stochastic}
\bibfield{author}{\bibinfo{person}{Mahmoud Assran}, \bibinfo{person}{Nicolas Loizou}, \bibinfo{person}{Nicolas Ballas}, {and} \bibinfo{person}{Mike Rabbat}.} \bibinfo{year}{2019}\natexlab{}.
\newblock \showarticletitle{Stochastic gradient push for distributed deep learning}. In \bibinfo{booktitle}{\emph{International Conference on Machine Learning}}. \bibinfo{pages}{344--353}.
\newblock


\bibitem[Bai et~al\mbox{.}(2024)]%
        {bai2024complexity}
\bibfield{author}{\bibinfo{person}{Yunyan Bai}, \bibinfo{person}{Yuxing Liu}, {and} \bibinfo{person}{Luo Luo}.} \bibinfo{year}{2024}\natexlab{}.
\newblock \showarticletitle{On the Complexity of Finite-Sum Smooth Optimization under the {P}olyak--{{\L}}ojasiewicz Condition}. In \bibinfo{booktitle}{\emph{International Conference on Machine Learning}}. \bibinfo{pages}{2392--2417}.
\newblock


\bibitem[Borodich and Kovalev(2025)]%
        {borodich2025nesterov}
\bibfield{author}{\bibinfo{person}{Ekaterina Borodich} {and} \bibinfo{person}{Dmitry Kovalev}.} \bibinfo{year}{2025}\natexlab{}.
\newblock \showarticletitle{Nesterov Finds GRAAL: Optimal and Adaptive Gradient Method for Convex Optimization}.
\newblock \bibinfo{journal}{\emph{arXiv preprint arXiv:2507.09823}} (\bibinfo{year}{2025}).
\newblock


\bibitem[Caldas et~al\mbox{.}(2018)]%
        {caldas2018leaf}
\bibfield{author}{\bibinfo{person}{Sebastian Caldas}, \bibinfo{person}{Sai Meher~Karthik Duddu}, \bibinfo{person}{Peter Wu}, \bibinfo{person}{Tian Li}, \bibinfo{person}{Jakub Kone{\v{c}}n{\`y}}, \bibinfo{person}{H~Brendan McMahan}, \bibinfo{person}{Virginia Smith}, {and} \bibinfo{person}{Ameet Talwalkar}.} \bibinfo{year}{2018}\natexlab{}.
\newblock \showarticletitle{{LEAF}: A benchmark for federated settings}.
\newblock \bibinfo{journal}{\emph{arXiv preprint arXiv:1812.01097}} (\bibinfo{year}{2018}).
\newblock


\bibitem[Chen et~al\mbox{.}(2023)]%
        {chen2023generalized}
\bibfield{author}{\bibinfo{person}{Ziyi Chen}, \bibinfo{person}{Yi Zhou}, \bibinfo{person}{Yingbin Liang}, {and} \bibinfo{person}{Zhaosong Lu}.} \bibinfo{year}{2023}\natexlab{}.
\newblock \showarticletitle{Generalized-smooth nonconvex optimization is as efficient as smooth nonconvex optimization}. In \bibinfo{booktitle}{\emph{International Conference on Machine Learning}}. \bibinfo{pages}{5396--5427}.
\newblock


\bibitem[Chezhegov et~al\mbox{.}(2025)]%
        {chezhegov2025convergence}
\bibfield{author}{\bibinfo{person}{Savelii Chezhegov}, \bibinfo{person}{Aleksandr Beznosikov}, \bibinfo{person}{Samuel Horv{\'a}th}, {and} \bibinfo{person}{Eduard Gorbunov}.} \bibinfo{year}{2025}\natexlab{}.
\newblock \showarticletitle{Convergence of Clipped-{SGD} for Convex {$(L_0, L_1)$}-Smooth Optimization with Heavy-Tailed Noise}.
\newblock \bibinfo{journal}{\emph{arXiv preprint arXiv:2505.20817}} (\bibinfo{year}{2025}).
\newblock


\bibitem[Cooper(2022)]%
        {cooper2024empirical}
\bibfield{author}{\bibinfo{person}{Y. Cooper}.} \bibinfo{year}{2022}\natexlab{}.
\newblock \showarticletitle{An empirical study of the {$(L_0, L_1)$}-smoothness condition}. In \bibinfo{booktitle}{\emph{Workshop on Mathematics of Modern Machine Learning}}.
\newblock


\bibitem[Crawshaw and Liu(2025)]%
        {crawshaw2025complexity}
\bibfield{author}{\bibinfo{person}{Michael Crawshaw} {and} \bibinfo{person}{Mingrui Liu}.} \bibinfo{year}{2025}\natexlab{}.
\newblock \showarticletitle{Complexity Lower Bounds of Adaptive Gradient Algorithms for Non-convex Stochastic Optimization under Relaxed Smoothness}.
\newblock \bibinfo{journal}{\emph{arXiv preprint arXiv:2505.04599}} (\bibinfo{year}{2025}).
\newblock


\bibitem[Crawshaw et~al\mbox{.}(2022)]%
        {crawshaw2022robustness}
\bibfield{author}{\bibinfo{person}{Michael Crawshaw}, \bibinfo{person}{Mingrui Liu}, \bibinfo{person}{Francesco Orabona}, \bibinfo{person}{Wei Zhang}, {and} \bibinfo{person}{Zhenxun Zhuang}.} \bibinfo{year}{2022}\natexlab{}.
\newblock \showarticletitle{Robustness to unbounded smoothness of generalized {SignSGD}}. In \bibinfo{booktitle}{\emph{Advances in Neural Information Processing Systems}}. \bibinfo{pages}{9955--9968}.
\newblock


\bibitem[Cutkosky and Mehta(2020)]%
        {cutkosky2020momentum}
\bibfield{author}{\bibinfo{person}{Ashok Cutkosky} {and} \bibinfo{person}{Harsh Mehta}.} \bibinfo{year}{2020}\natexlab{}.
\newblock \showarticletitle{Momentum improves normalized {SGD}}. In \bibinfo{booktitle}{\emph{International Conference on Machine Learning}}. \bibinfo{pages}{2260--2268}.
\newblock


\bibitem[Erd{\H{o}}s and R\'{e}nyi(1959)]%
        {erdds1959random}
\bibfield{author}{\bibinfo{person}{Paul Erd{\H{o}}s} {and} \bibinfo{person}{Alfr\'{e}d R\'{e}nyi}.} \bibinfo{year}{1959}\natexlab{}.
\newblock \showarticletitle{On random graphs {I}}.
\newblock \bibinfo{journal}{\emph{Publicationes Mathematicae Debrecen}} \bibinfo{volume}{6}, \bibinfo{number}{290-297} (\bibinfo{year}{1959}), \bibinfo{pages}{18}.
\newblock


\bibitem[Fang et~al\mbox{.}(2018)]%
        {fang2018spider}
\bibfield{author}{\bibinfo{person}{Cong Fang}, \bibinfo{person}{Chris~Junchi Li}, \bibinfo{person}{Zhouchen Lin}, {and} \bibinfo{person}{Tong Zhang}.} \bibinfo{year}{2018}\natexlab{}.
\newblock \showarticletitle{{SPIDER}: Near-optimal non-convex optimization via stochastic path-integrated differential estimator}. In \bibinfo{booktitle}{\emph{Advances in Neural Information Processing Systems}}. \bibinfo{pages}{687--697}.
\newblock


\bibitem[Gaash et~al\mbox{.}(2025)]%
        {gaash2025convergence}
\bibfield{author}{\bibinfo{person}{Ofir Gaash}, \bibinfo{person}{Kfir~Yehuda Levy}, {and} \bibinfo{person}{Yair Carmon}.} \bibinfo{year}{2025}\natexlab{}.
\newblock \showarticletitle{Convergence of Clipped {SGD} on Convex {$(L_0, L_1)$}-Smooth Functions}.
\newblock \bibinfo{journal}{\emph{arXiv preprint arXiv:2502.16492}} (\bibinfo{year}{2025}).
\newblock


\bibitem[Gorbunov et~al\mbox{.}(2021)]%
        {gorbunov2021local}
\bibfield{author}{\bibinfo{person}{Eduard Gorbunov}, \bibinfo{person}{Filip Hanzely}, {and} \bibinfo{person}{Peter Richt{\'a}rik}.} \bibinfo{year}{2021}\natexlab{}.
\newblock \showarticletitle{Local {SGD}: Unified theory and new efficient methods}. In \bibinfo{booktitle}{\emph{International Conference on Artificial Intelligence and Statistics}}. \bibinfo{pages}{3556--3564}.
\newblock


\bibitem[Gorbunov et~al\mbox{.}(2024)]%
        {gorbunov2024methods}
\bibfield{author}{\bibinfo{person}{Eduard Gorbunov}, \bibinfo{person}{Nazarii Tupitsa}, \bibinfo{person}{Sayantan Choudhury}, \bibinfo{person}{Alen Aliev}, \bibinfo{person}{Peter Richt{\'a}rik}, \bibinfo{person}{Samuel Horv{\'a}th}, {and} \bibinfo{person}{Martin Tak{\'a}{\v{c}}}.} \bibinfo{year}{2024}\natexlab{}.
\newblock \showarticletitle{Methods for convex {$(L_0, L_1)$}-smooth optimization: Clipping, acceleration, and adaptivity}.
\newblock \bibinfo{journal}{\emph{arXiv preprint arXiv:2409.14989}} (\bibinfo{year}{2024}).
\newblock


\bibitem[Hazan et~al\mbox{.}(2015)]%
        {hazan2015beyond}
\bibfield{author}{\bibinfo{person}{Elad Hazan}, \bibinfo{person}{Kfir Levy}, {and} \bibinfo{person}{Shai Shalev-Shwartz}.} \bibinfo{year}{2015}\natexlab{}.
\newblock \showarticletitle{Beyond convexity: Stochastic quasi-convex optimization}. In \bibinfo{booktitle}{\emph{Advances in Neural Information Processing Systems}}. \bibinfo{pages}{1594--1602}.
\newblock


\bibitem[Hendrikx et~al\mbox{.}(2021)]%
        {hendrikx2021optimal}
\bibfield{author}{\bibinfo{person}{Hadrien Hendrikx}, \bibinfo{person}{Francis Bach}, {and} \bibinfo{person}{Laurent Massoulie}.} \bibinfo{year}{2021}\natexlab{}.
\newblock \showarticletitle{An optimal algorithm for decentralized finite-sum optimization}.
\newblock \bibinfo{journal}{\emph{SIAM Journal on Optimization}} \bibinfo{volume}{31}, \bibinfo{number}{4} (\bibinfo{year}{2021}), \bibinfo{pages}{2753--2783}.
\newblock


\bibitem[Jiang et~al\mbox{.}(2025)]%
        {jiang2025decentralized}
\bibfield{author}{\bibinfo{person}{Zhanhong Jiang}, \bibinfo{person}{Aditya Balu}, {and} \bibinfo{person}{Soumik Sarkar}.} \bibinfo{year}{2025}\natexlab{}.
\newblock \showarticletitle{Decentralized Relaxed Smooth Optimization with Gradient Descent Methods}.
\newblock \bibinfo{journal}{\emph{arXiv preprint arXiv:2508.08413}} (\bibinfo{year}{2025}).
\newblock


\bibitem[Khirirat et~al\mbox{.}(2024a)]%
        {khirirat2024communication}
\bibfield{author}{\bibinfo{person}{Sarit Khirirat}, \bibinfo{person}{Abdurakhmon Sadiev}, \bibinfo{person}{Artem Riabinin}, \bibinfo{person}{Eduard Gorbunov}, {and} \bibinfo{person}{Peter Richt{\'a}rik}.} \bibinfo{year}{2024}\natexlab{a}.
\newblock \showarticletitle{Communication-efficient algorithms under generalized smoothness assumptions}.
\newblock \bibinfo{journal}{\emph{OpenReview}} (\bibinfo{year}{2024}).
\newblock


\bibitem[Khirirat et~al\mbox{.}(2024b)]%
        {khirirat2024error}
\bibfield{author}{\bibinfo{person}{Sarit Khirirat}, \bibinfo{person}{Abdurakhmon Sadiev}, \bibinfo{person}{Artem Riabinin}, \bibinfo{person}{Eduard Gorbunov}, {and} \bibinfo{person}{Peter Richt{\'a}rik}.} \bibinfo{year}{2024}\natexlab{b}.
\newblock \showarticletitle{Error Feedback under {$(L_0, L_1)$}-Smoothness: Normalization and Momentum}.
\newblock \bibinfo{journal}{\emph{arXiv preprint arXiv:2410.16871}} (\bibinfo{year}{2024}).
\newblock


\bibitem[Koloskova et~al\mbox{.}(2023)]%
        {koloskova2023revisiting}
\bibfield{author}{\bibinfo{person}{Anastasia Koloskova}, \bibinfo{person}{Hadrien Hendrikx}, {and} \bibinfo{person}{Sebastian~U. Stich}.} \bibinfo{year}{2023}\natexlab{}.
\newblock \showarticletitle{Revisiting gradient clipping: Stochastic bias and tight convergence guarantees}. In \bibinfo{booktitle}{\emph{International Conference on Machine Learning}}. \bibinfo{pages}{17343--17363}.
\newblock


\bibitem[Kovalev et~al\mbox{.}(2024)]%
        {kovalev2024lower}
\bibfield{author}{\bibinfo{person}{Dmitry Kovalev}, \bibinfo{person}{Ekaterina Borodich}, \bibinfo{person}{Alexander Gasnikov}, {and} \bibinfo{person}{Dmitrii Feoktistov}.} \bibinfo{year}{2024}\natexlab{}.
\newblock \showarticletitle{Lower Bounds and Optimal Algorithms for Non-Smooth Convex Decentralized Optimization over Time-Varying Networks}. In \bibinfo{booktitle}{\emph{Advances in Neural Information Processing Systems}}. \bibinfo{pages}{96566--96606}.
\newblock


\bibitem[Lan et~al\mbox{.}(2020)]%
        {lan2020communication}
\bibfield{author}{\bibinfo{person}{Guanghui Lan}, \bibinfo{person}{Soomin Lee}, {and} \bibinfo{person}{Yi Zhou}.} \bibinfo{year}{2020}\natexlab{}.
\newblock \showarticletitle{Communication-efficient algorithms for decentralized and stochastic optimization}.
\newblock \bibinfo{journal}{\emph{Mathematical Programming}} \bibinfo{volume}{180}, \bibinfo{number}{1} (\bibinfo{year}{2020}), \bibinfo{pages}{237--284}.
\newblock


\bibitem[LeCun et~al\mbox{.}(2002)]%
        {lecun2002gradient}
\bibfield{author}{\bibinfo{person}{Yann LeCun}, \bibinfo{person}{L{\'e}on Bottou}, \bibinfo{person}{Yoshua Bengio}, {and} \bibinfo{person}{Patrick Haffner}.} \bibinfo{year}{2002}\natexlab{}.
\newblock \showarticletitle{Gradient-based learning applied to document recognition}.
\newblock \bibinfo{journal}{\emph{Proc. IEEE}} \bibinfo{volume}{86}, \bibinfo{number}{11} (\bibinfo{year}{2002}), \bibinfo{pages}{2278--2324}.
\newblock


\bibitem[Li et~al\mbox{.}(2022)]%
        {li2022variance}
\bibfield{author}{\bibinfo{person}{Huan Li}, \bibinfo{person}{Zhouchen Lin}, {and} \bibinfo{person}{Yongchun Fang}.} \bibinfo{year}{2022}\natexlab{}.
\newblock \showarticletitle{Variance reduced {EXTRA} and {DIGing} and their optimal acceleration for strongly convex decentralized optimization}.
\newblock \bibinfo{journal}{\emph{Journal of Machine Learning Research}} \bibinfo{volume}{23}, \bibinfo{number}{222} (\bibinfo{year}{2022}), \bibinfo{pages}{1--41}.
\newblock


\bibitem[Li et~al\mbox{.}(2023)]%
        {li2023convergence}
\bibfield{author}{\bibinfo{person}{Haochuan Li}, \bibinfo{person}{Alexander Rakhlin}, {and} \bibinfo{person}{Ali Jadbabaie}.} \bibinfo{year}{2023}\natexlab{}.
\newblock \showarticletitle{Convergence of adam under relaxed assumptions}.
\newblock \bibinfo{journal}{\emph{Advances in Neural Information Processing Systems}} (\bibinfo{year}{2023}), \bibinfo{pages}{52166--52196}.
\newblock


\bibitem[Li et~al\mbox{.}(2024)]%
        {li2024problem}
\bibfield{author}{\bibinfo{person}{Jiaxiang Li}, \bibinfo{person}{Xuxing Chen}, \bibinfo{person}{Shiqian Ma}, {and} \bibinfo{person}{Mingyi Hong}.} \bibinfo{year}{2024}\natexlab{}.
\newblock \showarticletitle{Problem-parameter-free decentralized nonconvex stochastic optimization}.
\newblock \bibinfo{journal}{\emph{arXiv preprint arXiv:2402.08821}} (\bibinfo{year}{2024}).
\newblock


\bibitem[Lian et~al\mbox{.}(2017)]%
        {lian2017can}
\bibfield{author}{\bibinfo{person}{Xiangru Lian}, \bibinfo{person}{Ce Zhang}, \bibinfo{person}{Huan Zhang}, \bibinfo{person}{Cho-Jui Hsieh}, \bibinfo{person}{Wei Zhang}, {and} \bibinfo{person}{Ji Liu}.} \bibinfo{year}{2017}\natexlab{}.
\newblock \showarticletitle{Can decentralized algorithms outperform centralized algorithms? a case study for decentralized parallel stochastic gradient descent}. In \bibinfo{booktitle}{\emph{Advances in Neural Information Processing Systems}}. \bibinfo{pages}{5336--5346}.
\newblock


\bibitem[Lin et~al\mbox{.}(2024)]%
        {lin2023decentralized}
\bibfield{author}{\bibinfo{person}{Zhenwei Lin}, \bibinfo{person}{Jingfan Xia}, \bibinfo{person}{Qi Deng}, {and} \bibinfo{person}{Luo Luo}.} \bibinfo{year}{2024}\natexlab{}.
\newblock \showarticletitle{Decentralized Gradient-Free Methods for Stochastic Non-smooth Non-convex Optimization}. In \bibinfo{booktitle}{\emph{AAAI Conference on Artificial Intelligence}}. \bibinfo{pages}{17477--17486}.
\newblock


\bibitem[Liu and Morse(2011)]%
        {liu2011accelerated}
\bibfield{author}{\bibinfo{person}{Ji Liu} {and} \bibinfo{person}{A.~Stephen Morse}.} \bibinfo{year}{2011}\natexlab{}.
\newblock \showarticletitle{Accelerated linear iterations for distributed averaging}.
\newblock \bibinfo{journal}{\emph{Annual Reviews in Control}} \bibinfo{volume}{35}, \bibinfo{number}{2} (\bibinfo{year}{2011}), \bibinfo{pages}{160--165}.
\newblock


\bibitem[Liu et~al\mbox{.}(2024)]%
        {liu2024decentralized}
\bibfield{author}{\bibinfo{person}{Yuxing Liu}, \bibinfo{person}{Lesi Chen}, {and} \bibinfo{person}{Luo Luo}.} \bibinfo{year}{2024}\natexlab{}.
\newblock \showarticletitle{Decentralized Convex Finite-Sum Optimization with Better Dependence on Condition Numbers}. In \bibinfo{booktitle}{\emph{International Conference on Machine Learning}}. \bibinfo{pages}{30807--30841}.
\newblock


\bibitem[Liu and Zhou(2024)]%
        {liu2024nonconvex}
\bibfield{author}{\bibinfo{person}{Zijian Liu} {and} \bibinfo{person}{Zhengyuan Zhou}.} \bibinfo{year}{2024}\natexlab{}.
\newblock \showarticletitle{Nonconvex stochastic optimization under heavy-tailed noises: Optimal convergence without gradient clipping}.
\newblock \bibinfo{journal}{\emph{arXiv preprint arXiv:2412.19529}} (\bibinfo{year}{2024}).
\newblock


\bibitem[Lobanov et~al\mbox{.}(2024)]%
        {lobanov2024linear}
\bibfield{author}{\bibinfo{person}{Aleksandr Lobanov}, \bibinfo{person}{Alexander Gasnikov}, \bibinfo{person}{Eduard Gorbunov}, {and} \bibinfo{person}{Martin Tak{\'a}{\v{c}}}.} \bibinfo{year}{2024}\natexlab{}.
\newblock \showarticletitle{Linear Convergence Rate in Convex Setup is Possible! Gradient Descent Method Variants under {$(L_0, L_1)$}-Smoothness}.
\newblock \bibinfo{journal}{\emph{arXiv preprint arXiv:2412.17050}} (\bibinfo{year}{2024}).
\newblock


\bibitem[Lu and De~Sa(2021)]%
        {lu2021optimal}
\bibfield{author}{\bibinfo{person}{Yucheng Lu} {and} \bibinfo{person}{Christopher De~Sa}.} \bibinfo{year}{2021}\natexlab{}.
\newblock \showarticletitle{Optimal complexity in decentralized training}. In \bibinfo{booktitle}{\emph{International conference on machine learning}}. \bibinfo{pages}{7111--7123}.
\newblock


\bibitem[Luo et~al\mbox{.}(2022)]%
        {luo2022complexity}
\bibfield{author}{\bibinfo{person}{Luo Luo}, \bibinfo{person}{Yunyan Bai}, \bibinfo{person}{Lesi Chen}, \bibinfo{person}{Yuxing Liu}, {and} \bibinfo{person}{Haishan Ye}.} \bibinfo{year}{2022}\natexlab{}.
\newblock \showarticletitle{On the Complexity of Decentralized Finite-Sum Nonconvex Optimization}.
\newblock \bibinfo{journal}{\emph{arXiv preprint arXiv:2210.13931}} (\bibinfo{year}{2022}).
\newblock


\bibitem[Mandic(2004)]%
        {mandic2004generalized}
\bibfield{author}{\bibinfo{person}{Danilo~P. Mandic}.} \bibinfo{year}{2004}\natexlab{}.
\newblock \showarticletitle{A generalized normalized gradient descent algorithm}.
\newblock \bibinfo{journal}{\emph{IEEE Signal Processing Letters}} \bibinfo{volume}{11}, \bibinfo{number}{2} (\bibinfo{year}{2004}), \bibinfo{pages}{115--118}.
\newblock


\bibitem[Metelev et~al\mbox{.}(2024)]%
        {metelev2024decentralized}
\bibfield{author}{\bibinfo{person}{Dmitry Metelev}, \bibinfo{person}{Savelii Chezhegov}, \bibinfo{person}{Alexander Rogozin}, \bibinfo{person}{Aleksandr Beznosikov}, \bibinfo{person}{Alexander Sholokhov}, \bibinfo{person}{Alexander Gasnikov}, {and} \bibinfo{person}{Dmitry Kovalev}.} \bibinfo{year}{2024}\natexlab{}.
\newblock \showarticletitle{Decentralized finite-sum optimization over time-varying networks}.
\newblock \bibinfo{journal}{\emph{arXiv preprint arXiv:2402.02490}} (\bibinfo{year}{2024}).
\newblock


\bibitem[Nedi{\'c} et~al\mbox{.}(2018)]%
        {nedic2018network}
\bibfield{author}{\bibinfo{person}{Angelia Nedi{\'c}}, \bibinfo{person}{Alex Olshevsky}, {and} \bibinfo{person}{Michael~G Rabbat}.} \bibinfo{year}{2018}\natexlab{}.
\newblock \showarticletitle{Network topology and communication-computation tradeoffs in decentralized optimization}.
\newblock \bibinfo{journal}{\emph{Proc. IEEE}} \bibinfo{volume}{106}, \bibinfo{number}{5} (\bibinfo{year}{2018}), \bibinfo{pages}{953--976}.
\newblock


\bibitem[Nedic et~al\mbox{.}(2017)]%
        {nedic2017achieving}
\bibfield{author}{\bibinfo{person}{Angelia Nedic}, \bibinfo{person}{Alex Olshevsky}, {and} \bibinfo{person}{Wei Shi}.} \bibinfo{year}{2017}\natexlab{}.
\newblock \showarticletitle{Achieving geometric convergence for distributed optimization over time-varying graphs}.
\newblock \bibinfo{journal}{\emph{SIAM Journal on Optimization}} \bibinfo{volume}{27}, \bibinfo{number}{4} (\bibinfo{year}{2017}), \bibinfo{pages}{2597--2633}.
\newblock


\bibitem[Nedic and Ozdaglar(2009)]%
        {nedic2009distributed}
\bibfield{author}{\bibinfo{person}{Angelia Nedic} {and} \bibinfo{person}{Asuman Ozdaglar}.} \bibinfo{year}{2009}\natexlab{}.
\newblock \showarticletitle{Distributed subgradient methods for multi-agent optimization}.
\newblock \bibinfo{journal}{\emph{IEEE Trans. Automat. Control}} \bibinfo{volume}{54}, \bibinfo{number}{1} (\bibinfo{year}{2009}), \bibinfo{pages}{48--61}.
\newblock


\bibitem[Qu and Li(2017)]%
        {qu2017harnessing}
\bibfield{author}{\bibinfo{person}{Guannan Qu} {and} \bibinfo{person}{Na Li}.} \bibinfo{year}{2017}\natexlab{}.
\newblock \showarticletitle{Harnessing smoothness to accelerate distributed optimization}.
\newblock \bibinfo{journal}{\emph{IEEE Transactions on Control of Network Systems}} \bibinfo{volume}{5}, \bibinfo{number}{3} (\bibinfo{year}{2017}), \bibinfo{pages}{1245--1260}.
\newblock


\bibitem[Reisizadeh et~al\mbox{.}(2025)]%
        {reisizadeh2025variance}
\bibfield{author}{\bibinfo{person}{Amirhossein Reisizadeh}, \bibinfo{person}{Haochuan Li}, \bibinfo{person}{Subhro Das}, {and} \bibinfo{person}{Ali Jadbabaie}.} \bibinfo{year}{2025}\natexlab{}.
\newblock \showarticletitle{Variance-reduced clipping for non-convex optimization}. In \bibinfo{booktitle}{\emph{International Conference on Acoustics, Speech and Signal Processing}}. \bibinfo{pages}{1--5}.
\newblock


\bibitem[Sahinoglu and Shahrampour(2024)]%
        {sahinoglu2024online}
\bibfield{author}{\bibinfo{person}{Emre Sahinoglu} {and} \bibinfo{person}{Shahin Shahrampour}.} \bibinfo{year}{2024}\natexlab{}.
\newblock \showarticletitle{An Online Optimization Perspective on First-Order and Zero-Order Decentralized Nonsmooth Nonconvex Stochastic Optimization}. In \bibinfo{booktitle}{\emph{International Conference on Machine Learning}}. \bibinfo{pages}{43043--43059}.
\newblock


\bibitem[Scaman et~al\mbox{.}(2017)]%
        {scaman2017optimal}
\bibfield{author}{\bibinfo{person}{Kevin Scaman}, \bibinfo{person}{Francis Bach}, \bibinfo{person}{S{\'e}bastien Bubeck}, \bibinfo{person}{Yin~Tat Lee}, {and} \bibinfo{person}{Laurent Massouli{\'e}}.} \bibinfo{year}{2017}\natexlab{}.
\newblock \showarticletitle{Optimal algorithms for smooth and strongly convex distributed optimization in networks}. In \bibinfo{booktitle}{\emph{International Conference on Machine Learning}}. \bibinfo{pages}{3027--3036}.
\newblock


\bibitem[Song et~al\mbox{.}(2024)]%
        {song2024optimal}
\bibfield{author}{\bibinfo{person}{Zhuoqing Song}, \bibinfo{person}{Lei Shi}, \bibinfo{person}{Shi Pu}, {and} \bibinfo{person}{Ming Yan}.} \bibinfo{year}{2024}\natexlab{}.
\newblock \showarticletitle{Optimal gradient tracking for decentralized optimization}.
\newblock \bibinfo{journal}{\emph{Mathematical Programming}} \bibinfo{volume}{207}, \bibinfo{number}{1} (\bibinfo{year}{2024}), \bibinfo{pages}{1--53}.
\newblock


\bibitem[Tovmasyan et~al\mbox{.}(2025)]%
        {tovmasyan2025revisiting}
\bibfield{author}{\bibinfo{person}{Zhirayr Tovmasyan}, \bibinfo{person}{Grigory Malinovsky}, \bibinfo{person}{Laurent Condat}, {and} \bibinfo{person}{Peter Richt{\'a}rik}.} \bibinfo{year}{2025}\natexlab{}.
\newblock \showarticletitle{Revisiting Stochastic Proximal Point Methods: Generalized Smoothness and Similarity}.
\newblock \bibinfo{journal}{\emph{arXiv preprint arXiv:2502.03401}} (\bibinfo{year}{2025}).
\newblock


\bibitem[Tyurin(2025)]%
        {tyurin2024toward}
\bibfield{author}{\bibinfo{person}{Alexander Tyurin}.} \bibinfo{year}{2025}\natexlab{}.
\newblock \showarticletitle{Toward a unified theory of gradient descent under generalized smoothness}. In \bibinfo{booktitle}{\emph{International Conference on Learning Representations}}.
\newblock


\bibitem[Vankov et~al\mbox{.}(2025)]%
        {vankov2024optimizing}
\bibfield{author}{\bibinfo{person}{Daniil Vankov}, \bibinfo{person}{Anton Rodomanov}, \bibinfo{person}{Angelia Nedich}, \bibinfo{person}{Lalitha Sankar}, {and} \bibinfo{person}{Sebastian~U. Stich}.} \bibinfo{year}{2025}\natexlab{}.
\newblock \showarticletitle{Optimizing {$(L_0, L_1)$}-Smooth Functions by Gradient Methods}. In \bibinfo{booktitle}{\emph{International Conference on Learning Representations}}.
\newblock


\bibitem[Woodworth et~al\mbox{.}(2020)]%
        {woodworth2020minibatch}
\bibfield{author}{\bibinfo{person}{Blake~E. Woodworth}, \bibinfo{person}{Kumar~Kshitij Patel}, {and} \bibinfo{person}{Nati Srebro}.} \bibinfo{year}{2020}\natexlab{}.
\newblock \showarticletitle{Minibatch vs local {SGD} for heterogeneous distributed learning}. In \bibinfo{booktitle}{\emph{Advances in Neural Information Processing Systems}}. \bibinfo{pages}{6281--6292}.
\newblock


\bibitem[Xiao et~al\mbox{.}(2017)]%
        {xiao2017fashion}
\bibfield{author}{\bibinfo{person}{Han Xiao}, \bibinfo{person}{Kashif Rasul}, {and} \bibinfo{person}{Roland Vollgraf}.} \bibinfo{year}{2017}\natexlab{}.
\newblock \showarticletitle{Fashion-{MNIST}: a novel image dataset for benchmarking machine learning algorithms}.
\newblock \bibinfo{journal}{\emph{arXiv preprint arXiv:1708.07747}} (\bibinfo{year}{2017}).
\newblock


\bibitem[Xie et~al\mbox{.}(2024)]%
        {xie2024trust}
\bibfield{author}{\bibinfo{person}{Chenghan Xie}, \bibinfo{person}{Chenxi Li}, \bibinfo{person}{Chuwen Zhang}, \bibinfo{person}{Qi Deng}, \bibinfo{person}{Dongdong Ge}, {and} \bibinfo{person}{Yinyu Ye}.} \bibinfo{year}{2024}\natexlab{}.
\newblock \showarticletitle{Trust region methods for nonconvex stochastic optimization beyond lipschitz smoothness}. In \bibinfo{booktitle}{\emph{AAAI Conference on Artificial Intelligence}}. \bibinfo{pages}{16049--16057}.
\newblock


\bibitem[Xin et~al\mbox{.}(2020)]%
        {xin2020general}
\bibfield{author}{\bibinfo{person}{Ran Xin}, \bibinfo{person}{Shi Pu}, \bibinfo{person}{Angelia Nedi{\'c}}, {and} \bibinfo{person}{Usman~A. Khan}.} \bibinfo{year}{2020}\natexlab{}.
\newblock \showarticletitle{A general framework for decentralized optimization with first-order methods}.
\newblock \bibinfo{journal}{\emph{Proc. IEEE}} \bibinfo{volume}{108}, \bibinfo{number}{11} (\bibinfo{year}{2020}), \bibinfo{pages}{1869--1889}.
\newblock


\bibitem[Ye et~al\mbox{.}(2023)]%
        {ye2023multi}
\bibfield{author}{\bibinfo{person}{Haishan Ye}, \bibinfo{person}{Luo Luo}, \bibinfo{person}{Ziang Zhou}, {and} \bibinfo{person}{Tong Zhang}.} \bibinfo{year}{2023}\natexlab{}.
\newblock \showarticletitle{Multi-consensus decentralized accelerated gradient descent}.
\newblock \bibinfo{journal}{\emph{Journal of Machine Learning Research}} \bibinfo{volume}{24}, \bibinfo{number}{306} (\bibinfo{year}{2023}), \bibinfo{pages}{1--50}.
\newblock


\bibitem[Yu et~al\mbox{.}(2025)]%
        {yu2025mirror}
\bibfield{author}{\bibinfo{person}{Dingzhi Yu}, \bibinfo{person}{Wei Jiang}, \bibinfo{person}{Yuanyu Wan}, {and} \bibinfo{person}{Lijun Zhang}.} \bibinfo{year}{2025}\natexlab{}.
\newblock \showarticletitle{Mirror descent under generalized smoothness}.
\newblock \bibinfo{journal}{\emph{arXiv preprint arXiv:2502.00753}} (\bibinfo{year}{2025}).
\newblock


\bibitem[Yuan et~al\mbox{.}(2022)]%
        {yuan2022revisiting}
\bibfield{author}{\bibinfo{person}{Kun Yuan}, \bibinfo{person}{Xinmeng Huang}, \bibinfo{person}{Yiming Chen}, \bibinfo{person}{Xiaohan Zhang}, \bibinfo{person}{Yingya Zhang}, {and} \bibinfo{person}{Pan Pan}.} \bibinfo{year}{2022}\natexlab{}.
\newblock \showarticletitle{Revisiting optimal convergence rate for smooth and non-convex stochastic decentralized optimization}. In \bibinfo{booktitle}{\emph{Advances in Neural Information Processing Systems}}. \bibinfo{pages}{36382--36395}.
\newblock


\bibitem[Yuan et~al\mbox{.}(2016)]%
        {yuan2016convergence}
\bibfield{author}{\bibinfo{person}{Kun Yuan}, \bibinfo{person}{Qing Ling}, {and} \bibinfo{person}{Wotao Yin}.} \bibinfo{year}{2016}\natexlab{}.
\newblock \showarticletitle{On the convergence of decentralized gradient descent}.
\newblock \bibinfo{journal}{\emph{SIAM Journal on Optimization}} \bibinfo{volume}{26}, \bibinfo{number}{3} (\bibinfo{year}{2016}), \bibinfo{pages}{1835--1854}.
\newblock


\bibitem[Zhang et~al\mbox{.}(2020b)]%
        {zhang2020improved}
\bibfield{author}{\bibinfo{person}{Bohang Zhang}, \bibinfo{person}{Jikai Jin}, \bibinfo{person}{Cong Fang}, {and} \bibinfo{person}{Liwei Wang}.} \bibinfo{year}{2020}\natexlab{b}.
\newblock \showarticletitle{Improved analysis of clipping algorithms for non-convex optimization}. In \bibinfo{booktitle}{\emph{Advances in Neural Information Processing Systems}}. \bibinfo{pages}{15511--15521}.
\newblock


\bibitem[Zhang et~al\mbox{.}(2020a)]%
        {zhang2020gradient}
\bibfield{author}{\bibinfo{person}{Jingzhao Zhang}, \bibinfo{person}{Tianxing He}, \bibinfo{person}{Suvrit Sra}, {and} \bibinfo{person}{Ali Jadbabaie}.} \bibinfo{year}{2020}\natexlab{a}.
\newblock \showarticletitle{Why gradient clipping accelerates training: A theoretical justification for adaptivity}. In \bibinfo{booktitle}{\emph{International Conference on Learning Representations}}.
\newblock


\end{thebibliography}
